\documentclass[10pt]{article}
\usepackage{amsmath,amssymb}
\usepackage{enumerate}
\usepackage{fullpage}

\allowdisplaybreaks[1]

\newtheorem{theorem}{Theorem}[section]

\newtheorem{lemma}[theorem]{Lemma}
\newtheorem{proposition}[theorem]{Proposition}
\newtheorem{corollary}[theorem]{Corollary}
\newtheorem{definition}[theorem]{Definition}
\newtheorem{remark}[theorem]{Remark}
\newtheorem{example}[theorem]{Example}
\newtheorem{question}[theorem]{Question}

\newcommand{\im}{\operatorname{Im}}

\newcommand{\id}{\operatorname{id}}

\newcommand{\Sym}{\operatorname{Sym}}

\newcommand{\End}{\operatorname{End}}

\newcommand{\Aut}{\operatorname{Aut}}

\newcommand{\gr}{\operatorname{gr}}

\newcommand{\FM}{\operatorname{FM}}
\newcommand{\Map}{\operatorname{Map}}

\newenvironment{proof}{\par\noindent{\bf Proof.}}{$\qed$\par\bigskip}
\newcommand{\qed}{\enspace\vrule  height6pt  width4pt  depth2pt}
\usepackage{color}

\begin{document}

\title{Structure monoids of set-theoretic solutions of the Yang--Baxter equation\thanks{The first author was partially supported by the grants
MINECO-FEDER  MTM2017-83487-P and AGAUR 2017SGR1725 (Spain). The
second author is supported in part by Onderzoeksraad of Vrije
Universiteit Brussel and Fonds voor Wetenschappelijk Onderzoek
(Belgium). The third author is supported by Fonds voor Wetenschappelijk Onderzoek (Flanders), via an FWO Aspirant-mandate.}}
\author{Ferran Ced\'o \and Eric Jespers \and Charlotte Verwimp}
\date{}

\maketitle

\begin{abstract}
Given a set-theoretic solution $(X,r)$ of the Yang--Baxter equation,
we denote by $M=M(X,r)$ the structure monoid and by $A=A(X,r)$,
respectively $A'=A'(X,r)$, the left, respectively right, derived
structure monoid of $(X,r)$. It is shown that there exist a left
action of $M$ on $A$ and a right action of $M$ on $A'$ and
1-cocycles $\pi$ and $\pi'$ of $M$ with coefficients in $A$ and in
$A'$ with respect to these actions respectively. We investigate when
the 1-cocycles are injective, surjective or bijective.
 In case $X$ is finite, it turns out  that $\pi$ is bijective if and only
if $(X,r)$ is left non-degenerate, and $\pi'$ is bijective if and
only if $(X,r)$ is right non-degenerate.
In case $(X,r) $ is left non-degenerate, in particular $\pi$ is bijective, we  define a semi-truss structure on $M(X,r)$
and then we show that this naturally induces a set-theoretic solution $(\bar M, \bar r)$ on the least cancellative
image $\bar M= M(X,r)/\eta$ of $M(X,r)$. In case $X$ is naturally embedded in $M(X,r)/\eta$, for example when $(X,r)$ is irretractable, then $\bar r$ is an extension of $r$.
It also is shown that non-degenerate irretractable solutions necessarily are bijective.
\end{abstract}

\bigskip

\bigskip

\section{Introduction}

Let $V$ be a vector space over a field $K$. Solutions $R:V\otimes V
\rightarrow V\otimes V$ of the linear braid or Yang--Baxter equation
(abbreviated YBE)
$$ (R \otimes \id_V)\circ (\id_V \otimes R) \circ (R \otimes \id_V) =(\id_V  \otimes R)\circ (R \otimes \id_V)\circ (\id_V \otimes R)$$ on
the vector space $V\otimes V \otimes V$ have led  to several
algebraic structures, including  some classes of bialgebras, quantum
groups and Hopf algebras. Because the variety of solutions remains
elusive, Drinfeld \cite{drinfeld} in 1992 proposed  to consider
solutions that are linearizations  of solutions on a basis of $V$,
these are the so called set-theoretic solutions of the  YBE. Thus a
pair $(X,r)$, where $X$ is a non-empty set and $r: X \times X \rightarrow X
\times X$ is a  map, is called a set-theoretic solution of the YBE
if
$$ (r \times  \id_X)\circ (\id_X \times  r)\circ (r \times  \id_X) = (\id_X \times  r)\circ (r \times  \id_X)\circ (\id_X \times r).$$
For $x,y\in X$, write $r(x,y)=(\sigma_{x}(y),\gamma_{y}(x))$. The
solution $(X, r)$ is said to be  left (resp. right) non-degenerate
if each map $\sigma_x$ (resp. $\gamma_y$) is bijective. A left and
right non-degenerate solution is simply called a non-degenerate
solution. The solution $(X, r)$ is said to be involutive if $r^2 =
\id_{X\times X}$, in particular such a solution is bijective.

This study  started in the seminal papers of Etingof, Schedler and Soloviev \cite{ESS} and
Gateva--Ivanova and Van den Bergh \cite{GIV}. Since then, different aspects of this combinatorial
problem have been developed \cite{GIC,GIV,Lu,RumpAdv,S} and several interesting connections have
been found, such as braid and Garside groups \cite{Chouraqui,Dehornoy}, (semi)groups of I-type
\cite{GIV,JESOKN}, matched pairs of groups \cite{Lu,Takeuchi}, Artin--Schelter regular algebras
\cite{GIQUAD}, Jacobson radical rings and generalizations \cite{CJO, RumpBraces}, regular
subgroups and Hopf--Galois extensions \cite{SmoVen},
affine manifolds \cite{RumpClass}, orderability \cite{BCV,Chouraqui1} and
factorizable groups \cite{Wei}.

It is now well-known that all non-degenerate involutive set-theoretic solutions $(X,r)$ are restrictions of a set-theoretic solution on the structure monoid
 $$M(X,r) =\langle x\in X\mid xy=\sigma_x(y) \gamma_y(x), \text{ for all } x,y\in X \rangle .$$
Furthermore, in this case, the structure group
 $$G(X,r) =\gr ( x\in X\mid xy=\sigma_x(y) \gamma_y(x), \text{ for all } x,y \in X)$$
 and the permutation group $\mathcal{G}(X,r) = \gr ( \sigma_x \mid x\in X )$
have  a brace structure, an algebraic structure introduced by Rump
in \cite{RumpBraces}. Moreover, in \cite{BCJ}, it is shown that all
finite non-degenerate involutive set-theoretic solutions with a
given permutation group, as a brace, can be explicitly constructed.
For this case of finite solutions $(X,r)$, Etingof, Schedler and
Soloviev \cite{ESS} proved that $G(X,r)$ is a finitely generated,
solvable abelian-by-finite group and independently Gateva-Ivanova
and Van den Bergh \cite{GIV} have shown  that $G(X, r)$ is a
Bieberbach group, i.e. $G(X, r)$ is  an abelian-by-finite group,
torsion-free and finitely generated. To deal with arbitrary finite
bijective non-degenerate solutions Guarnieri and Vendramin \cite{GV}
introduced the algebraic structure called a skew brace. Bachiller
\cite{BSol} then also showed that all such solutions can be
described from  finite skew braces.  Lu, Yan and Zhu \cite{Lu}, and
Soloviev \cite{S} showed that for such solutions the structure group
$G(X,r)$ is abelian-by-finite (see also Lebed and Vendramin
\cite{LV} for another proof), and Jespers, Kubat and Van Antwerpen
\cite{JKV} showed  that the structure monoid $M(X,r)$ also is
abelian-by-finite. Note that, if $(X,r)$ is not involutive then  the
canonical map $i: X\rightarrow G(X,r)$ is not necessarily injective
and thus one cannot recover $r$ from the associated solution on
$G(X,r)$, however it can be recovered from the solution associated
to $M(X,r)$.

Also the associated structure algebras, i.e. the  monoid algebra $KM(X,r)$ and the group algebra $KG(X,r)$, where K is any field, have been studied by Jespers and Okni\'nski \cite{JESOKN},
 Gateva-Ivanova and Van den Bergh \cite{GIV}  and Jespers, Kubat and Van Antwerpen \cite{JKV}.
In the latter it is shown that if $(X,r)$ is a left non-degenerate
bijective finite set-theoretic solution then the algebra $KM(X,r)$
(and $KG(X,r)$) is a module-finite normal extension of a commutative
affine subalgebra. In particular, these algebras are Noetherian
PI-algebras of finite Gelfand-Kirillov dimension. Furthermore, it
was shown that many properties, such as being a  domain or  prime,
of the algebra $KM(X,r)$ are equivalent with the solution $(X, r)$
being involutive.

A crucial fact to prove the above results is (see \cite{ESS,JKV,Lu}) that
if $(X,r)$ is a  left non-degenerate solution then the structure monoid $M(X,r)$ is a regular submonoid of the semidirect product
 $$A(X,r) \rtimes \mathcal{G}(X,r),$$
 where
 $$A(X,r) =\langle x\in X \mid x\sigma_{x}(y) =\sigma_{x}(y) \sigma_{\sigma_{x}(y)}(\gamma_{y}(x)), \text{ for all } x,y \in X \rangle , $$
 i.e.
 for any element $a\in A(X,r)$ there is a unique $\phi (a)\in \mathcal{G}(X,r)$ such that
 $(a,\phi (a))\in M(X,r)$. In particular, one has a bijective 1-cocycle $M(X,r) \rightarrow A(X,r)$,
 determined by the natural action of $\mathcal{G}(X,r)$ on $A(X,r)$.
 Here the derived monoid $A(X,r)$ ``encodes'' the relations determined by the map $r^{2}:X^{2}
 \rightarrow X^{2}$.
If, furthermore, the left non-degenerate solution $(X,r)$ is
bijective then  the monoid $A=A(X,r)$ is such that $aA=Aa$, for all
$a\in A$. So $A(X,r)$  consists of normal elements and thus $A$ has
a much richer structure than $M(X,r)$.  For example, if $(X,r)$ is
involutive then $A$ is a free abelian monoid of rank $|X|$. It is
this  ``richer structure'' that has been exploited in several papers
to obtain information on the structure monoid $M(X,r)$ and the
structure algebra $KM(X,r)$.

In this paper we continue these investigations for arbitrary
set-theoretic solutions $(X,r)$. So, $r$ is not necessarily
bijective and $X$ is any set. In the first section we recall the
important result of Gateva-Ivanova and Majid \cite{GIM2008}: there
exists a unique set-theoretic solution $(M,r_M)$ associated to the
structure monoid $M=M(X,r)$ such that the restriction of $r_M$ to
$X^2$ equals $r$. In the second section we introduce two derived
monoids $A(X,r)$ and $A'(X,r)$ and we prove that there is a  unique
$1$-cocycle $\pi : M(X,r) \rightarrow A(X,r)$, with respect to the
natural left action $\lambda': M(X,r)\rightarrow
\text{End}(A(X,r))$, such that $\pi (x)=x$, and a unique 1-cocycle
$\pi' : M(X,r) \rightarrow A'(X,r)$, with respect to natural right
action $\rho' :M(X,r)\rightarrow \text{End}(A'(X,r))$ such that
$\pi'(x)=x$. Hence one gets a monoid homomorphism $f: M(X,r)
\rightarrow A(X,r) \rtimes \im(\lambda'): a\mapsto (\pi (a) ,
\lambda_a')$ and a monoid anti-homomorphism $f': M(X,r) \rightarrow
A'(X,r)^{op} \rtimes \im(\rho'): a\mapsto (\pi'(a),\rho'_a) $, where
$\lambda'_x(y)=\sigma_x(y)$ and $\rho'_x(y)=\gamma_x(y)$, for all
$x,y\in X$. In general these $1$-cocycles are not bijective. But we
investigate when they are respectively injective, respectively
surjective. In case $(X,r)$ is finite the bijectiveness of $\pi$
(respectively $\pi'$) is equivalent with the solution being left
(respectively right) non-degenerate. In Section 4 we prove the
surprising result that any non-degenerate irretractable solution is
necessarily bijective. In Section~5 we  link the algebraic structure
of $M(X,r)$ to that of semi-trusses as introduced by Brzezi\'nski
\cite{Br}. We determine the left cancellative (additive) congruence
$\eta$ on $M(X,r)$ for $(X,r)$ a left non-degenerate solution, and
we  show that we obtain a solution $(M/\eta , \overline{r})$
determined by a semi-truss structure on $M/\eta$.

\section{Solution associated with the structure monoid}
\label{sectiontwo}

 In this section we recall a result of
Gateva-Ivanova and Majid in \cite[Section~3.2]{GIM2008} stating that
any set-theoretic solution $(X,r)$  of the YBE can be extended
to a set-theoretic solution on its structure monoid $M(X,r)$. The
result in \cite{GIM2008} is stated  for bijective solutions but the
proof remains valid without this assumption.

Recall this construction. Let $(X,r)$ be a set-theoretic solution of
the YBE which is not necessarily bijective. We write
$r(x,y)=(\sigma_x(y),\gamma_y(x))$, for all $x,y\in X$. It is known
that $(X,r)$ is a set-theoretic solution of the YBE if and only if
the following conditions hold:
    \begin{eqnarray}
    \sigma_x\sigma_y &=& \sigma_{\sigma_x(y)} \sigma_{\gamma_y(x)},
    \label{YBE1}
    \\ \sigma_{\gamma_{\sigma_{x}(y)}(z)}(\gamma_y(x)) &=& \gamma_{\sigma_{\gamma_x(z)}(y)}(\sigma_z(x)), \label{YBE2}
    \\ \gamma_x\gamma_y &=& \gamma_{\gamma_x(y)}\gamma_{\sigma_y(x)}, \label{YBE3}
    \end{eqnarray}
for all $x,y,z\in X$.

Let $M=M(X,r)$ be the structure monoid of $(X,r)$, that is the multiplicative
monoid with operation $\circ$ and  with presentation
$$M(X,r)=\langle X \mid x\circ y=\sigma_x(y)\circ \gamma_y(x), \mbox{ for all }x,y\in X\rangle.$$
One  defines the following ``left action'' on $M$:
$$\lambda\colon M\longrightarrow \Map(M,M): a\mapsto \lambda_a, $$
with
  $\lambda_1=\id_{M}$,
and for
$x_1,\dots ,x_m,y_1,\dots ,y_n\in X$ and $n>1$,
 $\lambda_{x_1}(1)=1$,
\begin{equation}\label{lambda1}
 \lambda_{x_1}(y_1)=\sigma_{x_1}(y_1),\;\;
\lambda_{x_1}(y_1 \circ \cdots \circ y_n)=\sigma_{x_1}(y_1)\circ
\lambda_{\gamma_{y_1}(x_1)}(y_2\circ \cdots \circ y_n),
\end{equation}
and for $m>1$,
\begin{equation}\label{lambda2}
\lambda_{x_1\circ \dots \circ x_m}=\lambda_{x_1} \circ \cdots \circ
\lambda_{x_m}.
\end{equation}
One also defines a ``right action'' on $M$:
 $$\rho\colon M\longrightarrow \Map(M,M): a\mapsto \rho_a,$$
with $\rho_1=\id_{M}$,  and for
$x_1,\dots ,x_m,y_1,\dots ,y_n\in X$ and $n>1$,
\begin{equation}\label{erho1}
\rho_{x_1}(y_1)=\gamma_{x_1}(y_1),\;\;
\rho_{x_1}(y_1\circ \cdots \circ
y_n)=\rho_{\sigma_{y_n}(x_1)}(y_1\circ \cdots \circ y_{n-1})\circ
\gamma_{x_1}(y_n),\end{equation}
and for $m>1$
\begin{equation}\label{erho2}
\rho_{x_1\circ \dots \circ x_m}=\rho_{x_m}\circ \cdots \circ  \rho_{x_1}.
\end{equation}

In \cite{GIM2008} it is proved that $\lambda$ and $\rho$ are well
defined.  Furthermore,  it is then shown that every set-theoretic
solution $(X,r)$ of the YBE is the restriction of a
set-theoretic solution defined on the structure monoid $M(X,r)$.

\begin{theorem}\cite[Theorem 3.6]{GIM2008} (Gateva-Ivanova and Majid)\label{main}
Let $(X,r)$ be a set-theoretic solution of the YBE. Then
the mapping $\lambda$ is a monoid homomorphism and the mapping
$\rho$ is monoid anti-homomorphism such that
\begin{align}
\rho_b(c\circ a) &= \rho_{\lambda_a(b)}(c)\circ \rho_b(a), \label{eqrho} \\
\lambda_b(a \circ c) &= \lambda_b(a) \circ \lambda_{\rho_a(b)}(c),
\label{eqlambda}
\end{align} for all $a,b,c\in M$.
Furthermore, for $a,b\in M=M(X,r)$,
\begin{equation}\label{eqproduct}
a\circ b=\lambda_{a}(b)\circ\rho_{b}(a).
\end{equation}
Let $r_M\colon M\times M\rightarrow M\times M$ be defined by
$r_M(a,b)=(\lambda_a(b),\rho_b(a))$, for all $a,b\in M$. Then,
$(M,r_M)$ is a set-theoretic solution of the YBE. Obviously, $r_M$
extends the solution $r$.
\end{theorem}

Note that if the solution $(X,r)$ is bijective and left and right non-degenerate, i.e. all $\sigma_x$ and $\gamma_x$
are bijective maps, then as in the proof of the above result one can show that the mappings $\sigma_x$ and $\gamma_x$ induce actually left and right actions
on $G=G(X,r)$, say
 $\lambda^{e} : G\rightarrow \Sym (G)$
and $\rho^{e}: G\rightarrow \Sym (G)$; this is Theorem 4 in \cite{Lu}. Furthermore,
the mapping $r_G (a,b) =(\lambda^{e}_a(b),\rho^{e}_b(a))$, for $a,b\in G$,
defines  a set-theoretic solution on $G$.
Note that, in general, the natural map $i:X\rightarrow G$ is not injective.
One obtains that $r_G$ is an extension of the
induced set-theoretic solution
$(i(X),r_{i(X)^{2}})= (i(X),(r_G)_{i(X)^{2}})$.

A natural question is whether one can extend a solution $(X,r)$, via the actions induced from $\sigma_x$ and $\gamma_y$,  to a solution on the structure group.
This however is not possible in general as shown by the following example.
Consider the set-theoretic solution  $(X,\id_{X^{2}})$ on a set $X$ with more than one element. Obviously,
each $\sigma_x $ and $\gamma_x$  is the constant with image $\{ x \}$. Hence, $M=M(X,\id_{X^{2}})$ is the free  monoid
on the set $X$ and $G=G(X,\id_{X^{2}})$ is the free group on $X$.
However, because the  maps $\sigma_x$ are not injective one cannot extend the maps $\sigma_x$ to a monoid homomorphism $\lambda : G\rightarrow \Map (G,G)$
with   $\lambda_x (y)=\sigma_x(y)$, for $y\in G$.

A remarkable fact shown by Lu, Yan and Zhu in \cite{Lu} is that if
one can extend the mappings $\sigma_x$ and $\gamma_x$ to left and right actions on
the structure group then the induced set-theoretic solution is
bijective.

\section{Derived monoids}\label{derived}

Let $(X,r)$ be a set-theoretic solution of the YBE. Write
$r(x,y)=(\sigma_x(y), \gamma_y(x))$,  for all $x,y\in X$. If $(X,r)$
is left non-degenerate, then Soloviev defined in \cite{S} its
derived solution $(X,r')$ by
$$r'(x,y)=(y,\sigma_y\gamma_{\sigma^{-1}_x(y)}(x)),$$
for all $x,y\in X$.
For general solutions one cannot define such a
derived solution. But in \cite{JKV} one   defines the
derived monoids of $(X,r)$ as
$$A(X,r)=\langle X\mid x+\sigma_x(y)=\sigma_x(y)+\sigma_{\sigma_x(y)}\gamma_y(x), \text{ for all }x,y\in X\rangle $$
and
$$A'(X,r)=\langle X\mid \gamma_y(x)\oplus y=\gamma_{\gamma_y(x)}\sigma_x(y)\oplus\gamma_y(x), \text{ for all } x,y\in X\rangle.$$
The zero element of $A(X,r)$ is denoted $0$ and the zero element of $A'(X,r)$ is denoted $0'$. We will say that $A(X,r)$ is the {\em left derived} structure
monoid of $(X,r)$ and $A'(X,r)$ is the {\em right derived} structure
monoid of $(X,r)$.

Note that $X\subseteq M(X,r)$, $X\subseteq A(X,r)$ and $X\subseteq
A'(X,r)$, because the defining relations of these three monoids are
homogeneous of degree $2$.

\begin{proposition}\label{lambda'rho'}
Let $(X,r)$ be a set-theoretic solution of the YBE, where
$r(x,y)=(\sigma_x(y), \gamma_y(x))$, for all $x,y \in X$. Then
there exists a unique monoid homomorphism $\lambda'\colon
M(X,r)\longrightarrow \End(A(X,r))$ such that,
$\lambda'(x)(y)=\sigma_x(y)$, for all $x,y\in X$ and  there exists a
unique anti-homomorphism $\rho'\colon  M(X,r)\longrightarrow
\End(A'(X,r))$ such that, $\rho'(x)(y)=\gamma_x(y)$, for all $x,y\in
X$. Furthermore, if $(X,r)$ is left (right) non-degenerate, then $\im(\lambda')\subseteq\Aut(A(X,r))$ ($\im(\rho')\subseteq\Aut(A'(X,r))$).\end{proposition}

\begin{proof}
We will write $\lambda'(a)=\lambda'_a$ and $\rho'(a)=\rho'_a$, for
all $a\in M(X,r)$.

Let $x_1,\dots, x_m,y_1,\dots ,y_n\in X$. We denote by $1,0,0'$ the
identity elements of the monoids $M(X,r)$, $A(X,r)$, $A'(X,r)$,
respectively. We define $\lambda'_1=\id_{A(X,r)}$,
$\rho'_1=\id_{A'(X,r)}$, $\lambda'_a(0)=0$, $\rho'_a(0')=0'$, for
all $a\in M(X,r)$, and
\begin{align*}
\lambda'_{ x_1\circ\cdots\circ x_m}(y_1+\dots +y_n)
&=\sigma_{x_1}\dots \sigma_{x_m}(y_1)+\dots +\sigma_{x_1}\dots
\sigma_{x_m}(y_n),
\end{align*}
and
\begin{align*}
\rho'_{ x_1\circ\cdots\circ x_m}(y_1\oplus\dots\oplus y_n)
&=\gamma_{x_m}\dots
\gamma_{x_1}(y_1)\oplus\dots\oplus\gamma_{x_m}\dots
\gamma_{x_1}(y_n),
\end{align*}

First we shall prove that $\lambda'$ and $\rho'$ are well-defined.
To do so it is enough to prove that the following equalities hold:
\begin{align}
\lambda'_{x_1 \circ x_2}(y_1 + \cdots + y_n) &= \lambda'_{\sigma_{x_1}(x_2) \circ \gamma_{x_2}(x_1)}(y_1 + \cdots + y_n), \label{eq:lambda'1} \\
\lambda'_{x_1 \circ \cdots \circ x_m}(y_1 + \sigma_{y_1}(y_2)) &= \lambda'_{x_1 \circ \cdots \circ x_m}(\sigma_{y_1}(y_2) + \sigma_{\sigma_{y_1}(y_2)}(\gamma_{y_2}(y_1))),\label{eq:lambda'2} \\
\rho'_{x_1 \circ x_2}(y_1 \oplus \cdots \oplus y_n) &= \rho'_{\sigma_{x_1}(x_2) \circ \gamma_{x_2}(x_1)}(y_1 \oplus \cdots \oplus y_n), \label{eq:rho'1} \\
\rho'_{x_1 \circ \cdots \circ x_m}(\gamma_{y_2}(y_1) \oplus y_2) &=
\rho'_{x_1 \circ \cdots \circ
x_m}(\gamma_{\gamma_{y_2}(y_1)}(\sigma_{y_1}(y_2)) \oplus
\gamma_{y_2}(y_1)). \label{eq:rho'2}
\end{align}

Using the relations (\ref{YBE1}) and (\ref{YBE3}), equations (\ref{eq:lambda'1})
and (\ref{eq:rho'1}) are easily checked:
\begin{align*}
\lambda'_{x_1 \circ x_2}(y_1 + \cdots + y_n)
&=\sigma_{x_1}\sigma_{x_2}(y_1) + \cdots +
\sigma_{x_1}\sigma_{x_2}(y_n)
\\ &= \sigma_{\sigma_{x_1}(x_2)}\sigma_{\gamma_{x_2}(x_1)}(y_1) + \cdots + \sigma_{\sigma_{x_1}(x_2)}\sigma_{\gamma_{x_2}(x_1)}(y_n)
\\ &= \lambda'_{\sigma_{x_1}(x_2) \circ \gamma_{x_2}(x_1)}(y_1 + \cdots + y_n), \\
\rho'_{x_1 \circ x_2}(y_1 \oplus \cdots \oplus y_n) &=
\gamma_{x_2}\gamma_{x_1}(y_1) \oplus \cdots \oplus
\gamma_{x_2}\gamma_{x_1}(y_n)
\\ &= \gamma_{\gamma_{x_2}(x_1)}\gamma_{\sigma_{x_1}(x_2)}(y_1) \oplus \cdots \oplus \gamma_{\gamma_{x_2}(x_1)}\gamma_{\sigma_{x_1}(x_2)}(y_n)
\\ &= \rho'_{\sigma_{x_1}(x_2) \circ \gamma_{x_2}(x_1)}(y_1 \oplus \cdots \oplus y_n).
\end{align*}

Using the relations (\ref{YBE1}), (\ref{YBE2}) and (\ref{YBE3}) we
shall prove equations (\ref{eq:lambda'2}) and (\ref{eq:rho'2})  by
induction on $m$. For $m=0$, (\ref{eq:lambda'2}) and
(\ref{eq:rho'2}) follows by the defining relations of $A(X,r)$ and
$A'(X,r)$. Suppose that $m>0$. Assume that $\lambda'_{x_1 \circ
\cdots \circ x_k}(y_1 + \sigma_{y_1}(y_2)) = \lambda'_{x_1 \circ
\cdots \circ x_k}(\sigma_{y_1}(y_2) +
\sigma_{\sigma_{y_1}(y_2)}(\gamma_{y_2}(y_1)))$ and $\rho'_{x_1
\circ \cdots \circ x_k}(\gamma_{y_2}(y_1) \oplus y_2) = \rho'_{x_1
\circ \cdots \circ
x_k}(\gamma_{\gamma_{y_2}(y_1)}(\sigma_{y_1}(y_2)) \oplus
\gamma_{y_2}(y_1))$, for $k<m$, then
\begin{align*}
& \lambda'_{x_1 \circ \cdots \circ x_m}(y_1 + \sigma_{y_1}(y_2))
\\ &= \sigma_{x_1}\cdots \sigma_{x_m}(y_1) + \sigma_{x_1}\cdots \sigma_{x_m}(\sigma_{y_1}(y_2))
\\ &= \lambda'_{x_1 \circ \cdots \circ x_{m-1}}(\sigma_{x_m}(y_1) + \sigma_{x_m}(\sigma_{y_1}(y_2)))
\\ &= \lambda'_{x_1 \circ \cdots \circ x_{m-1}}(\sigma_{x_m}(y_1) + \sigma_{\sigma_{x_m}(y_1)}(\sigma_{\gamma_{y_1}(x_m)}(y_2)))
\\ &= \lambda'_{x_1 \circ \cdots \circ x_{m-1}}(\sigma_{\sigma_{x_m}(y_1)}(\sigma_{\gamma_{y_1}(x_m)}(y_2))
+
\sigma_{\sigma_{\sigma_{x_m}(y_1)}(\sigma_{\gamma_{y_1}(x_m)}(y_2))}(\gamma_{\sigma_{\gamma_{y_1}(x_m)}(y_2)}(\sigma_{x_m}(y_1))))
\\ &= \lambda'_{x_1 \circ \cdots \circ x_{m-1}}(\sigma_{x_m}(\sigma_{y_1}(y_2)) + \sigma_{\sigma_{x_m}(\sigma_{y_1}(y_2))}(\gamma_{\sigma_{\gamma_{y_1}(x_m)}(y_2)}(\sigma_{x_m}(y_1))))
\\ &= \lambda'_{x_1 \circ \cdots \circ x_{m-1}}(\sigma_{x_m}(\sigma_{y_1}(y_2)) + \sigma_{\sigma_{x_m}(\sigma_{y_1}(y_2))}(\sigma_{\gamma_{\sigma_{y_1}(y_2)}(x_m)}(\gamma_{y_2}(y_1))))
\\ &= \lambda'_{x_1 \circ \cdots \circ x_{m-1}}(\sigma_{x_m}(\sigma_{y_1}(y_2)) + \sigma_{x_m}(\sigma_{\sigma_{y_1}(y_2)}(\gamma_{y_2}(y_1))))
\\ &= \sigma_{x_1} \cdots \sigma_{x_m}(\sigma_{y_1}(y_2)) + \sigma_{x_1} \cdots \sigma_{x_m}(\sigma_{\sigma_{y_1}(y_2)}(\gamma_{y_2}(y_1)))
\\ &= \lambda'_{x_1 \circ \cdots \circ x_m}(\sigma_{y_1}(y_2) + \sigma_{\sigma_{y_1}(y_2)}(\gamma_{y_2}(y_1))),
\end{align*}
and
\begin{align*}
& \rho'_{x_1\circ \cdots \circ x_m}(\gamma_{y_2}(y_1) \oplus y_2)
\\ &= \rho'_{x_2\circ \cdots \circ x_m}(\gamma_{x_1}(\gamma_{y_2}(y_1)) \oplus \gamma_{x_1}(y_2))
\\ &= \rho'_{x_2\circ \cdots \circ x_m}(\gamma_{\gamma_{x_1}(y_2)}(\gamma_{\sigma_{y_2}(x_1)}(y_1)) \oplus \gamma_{x_1}(y_2))
\\ &= \rho'_{x_2\circ \cdots \circ x_m}(\gamma_{\gamma_{\gamma_{x_1}(y_2)}(\gamma_{\sigma_{y_2}(x_1)}(y_1))}(\sigma_{\gamma_{\sigma_{y_2}(x_1)}(y_1)}(\gamma_{x_1}(y_2)))
\oplus \gamma_{\gamma_{x_1}(y_2)}(\gamma_{\sigma_{y_2}(x_1)}(y_1)))
\\ &= \rho'_{x_2\circ \cdots \circ x_m}(\gamma_{\gamma_{x_1}(\gamma_{y_2}(y_1))}(\sigma_{\gamma_{\sigma_{y_2}(x_1)}(y_1)}(\gamma_{x_1}(y_2))) \oplus \gamma_{x_1}(\gamma_{y_2}(y_1)))
\\ &= \rho'_{x_2\circ \cdots \circ x_m}(\gamma_{\gamma_{x_1}(\gamma_{y_2}(y_1))}(\gamma_{\sigma_{\gamma_{y_2}(y_1)}(x_1)}(\sigma_{y_1}(y_2))) \oplus \gamma_{x_1}(\gamma_{y_2}(y_1)))
\\ &= \rho'_{x_2\circ \cdots \circ x_m}(\gamma_{x_1}(\gamma_{\gamma_{y_2}(y_1)}(\sigma_{y_1}(y_2))) \oplus \gamma_{x_1}(\gamma_{y_2}(y_1)))
\\ &= \gamma_{x_m}\cdots \gamma_{x_1}(\gamma_{\gamma_{y_2}(y_1)}(\sigma_{y_1}(y_2))) \oplus \gamma_{x_m}\cdots\gamma_{x_1}(\gamma_{y_2}(y_1))
\\ &= \rho'_{x_1\circ \cdots \circ x_m}(\gamma_{\gamma_{y_2}(y_1)}(\sigma_{y_1}(y_2)) \oplus \gamma_{y_2}(y_1)).
\end{align*}
This proves that $\lambda'_a$ and $\rho'_a$ are well-defined and
clearly $\lambda'_a\in \End(A(X,r))$ and $\rho'_a\in \End(A'(X,r))$,
for all $a\in M(X,r)$. Thus $\lambda'$ and $\rho'$ are well-defined.
It is clear that $\lambda'$ is a monoid  homomorphism and that it is unique
with respect to the condition
$\lambda'_x(y)=\sigma_x(y)$, for all $x,y\in X$. It also is
clear that $\rho'$ is a monoid anti-homomorphism and that it is
unique for the condition $\rho'_x(y)=\gamma_x(y)$, for all $x,y\in X$.

Assume now that $(X,r)$ is left non-degenerate. Let
$x,y_1,\dots ,y_n\in X$.  We define $f_x\in \End(A(X,r))$ by
\begin{align*}
f_{x}(y_1+\dots +y_n) &=\sigma^{-1}_{x}(y_1)+\dots
+\sigma^{-1}_{x}(y_n).
\end{align*}
To see that $f_x$ is well-defined it is enough to prove that
$$f_{x}(y_1 + \sigma_{y_1}(y_2)) = f_{x}(\sigma_{y_1}(y_2) + \sigma_{\sigma_{y_1}(y_2)}(\gamma_{y_2}(y_1))).$$
Note that, from (\ref{YBE1}),
  \begin{eqnarray}
   \sigma_x^{-1} \sigma_{y_{1}} (y_2) &= &\sigma_{\sigma_x^{-1} (y_1)} \sigma^{-1}_{\gamma_{\sigma_x^{-1} (y_1)}(x)}(y_2) \label{useful}
   \end{eqnarray}
  and thus, also using (\ref{YBE2}), we get that
\begin{eqnarray}\label{neweq}
\lefteqn{\sigma_{\gamma_{\sigma^{-1}_{x}\sigma_{y_1}(y_2)}(x)}\gamma_{\sigma^{-1}_{\gamma_{\sigma^{-1}_{x}(y_1)}(x)}(y_2)}(\sigma^{-1}_{x}(y_1))}\\
&&=\sigma_{\gamma_{\sigma_{\sigma^{-1}_{x}(y_1)}\sigma^{-1}_{\gamma_{\sigma^{-1}_x(y_1)}(x)}(y_2)}(x)}
\gamma_{\sigma^{-1}_{\gamma_{\sigma^{-1}_{x}(y_1)}(x)}(y_2)}(\sigma^{-1}_{x}(y_1))\notag\\
&&=\gamma_{\sigma_{\gamma_{\sigma^{-1}_{x}(y_1)}(x)}(\sigma^{-1}_{\gamma_{\sigma^{-1}_x(y_1)}(x)}(y_2))}\sigma_{x}(\sigma^{-1}_{x}(y_1))\notag\\
&&=\gamma_{y_2}(y_1).\notag
\end{eqnarray}
We have that
\begin{align*}
& f_{x}(y_1 + \sigma_{y_1}(y_2))
\\ &= \sigma^{-1}_{x}(y_1) + \sigma^{-1}_{x}(\sigma_{y_1}(y_2))
\\ &\stackrel{(\ref{useful})}{=} \sigma^{-1}_{x}(y_1) + \sigma_{\sigma^{-1}_{x}(y_1)}(\sigma^{-1}_{\gamma_{\sigma^{-1}_{x}(y_1)}(x)}(y_2))
\\ &= \sigma_{\sigma^{-1}_{x}(y_1)}(\sigma^{-1}_{\gamma_{\sigma^{-1}_{x}(y_1)}(x)}(y_2))+
\sigma_{\sigma_{\sigma^{-1}_{x}(y_1)}(\sigma^{-1}_{\gamma_{\sigma^{-1}_{x}(y_1)}(x)}(y_2))}(\gamma_{\sigma^{-1}_{\gamma_{\sigma^{-1}_{x}(y_1)}(x)}(y_2)}(\sigma^{-1}_{x}(y_1)))
\\ &\stackrel{(\ref{useful})}{=}\sigma^{-1}_{x}(\sigma_{y_1}(y_2)) + \sigma_{\sigma^{-1}_{x}(\sigma_{y_1}(y_2))}(\gamma_{\sigma^{-1}_{\gamma_{\sigma^{-1}_{x}(y_1)}(x)}(y_2)}(\sigma^{-1}_{x}(y_1)))
\\ &\stackrel{
(\ref{neweq})}{=} \sigma^{-1}_{x}(\sigma_{y_1}(y_2))  +
\sigma_{\sigma^{-1}_{x}(\sigma_{y_1}(y_2))}(\sigma^{-1}_{\gamma_{\sigma^{-1}_{x}\sigma_{y_1}(y_2)}(x)}(\gamma_{y_2}(y_1)))
\\ &\stackrel{(\ref{useful})}{=} \sigma^{-1}_{x}(\sigma_{y_1}(y_2))  + \sigma^{-1}_{x}(\sigma_{\sigma_{y_1}(y_2)}(\gamma_{y_2}(y_1)))
\\ &= f_{x}(\sigma_{y_1}(y_2) +
\sigma_{\sigma_{y_1}(y_2)}(\gamma_{y_2}(y_1))),
\end{align*}
where the third equality follows from the defining relations in $A(X,r)$.
Hence $f_x$ is well-defined. Note that
$f_x\lambda'_x=\lambda'_xf_x=\id$. Thus $\lambda'_x\in\Aut(A(X,r))$,
for all $x\in X$. Therefore $\im(\lambda')\subseteq\Aut(A(X,r))$.

Similarly one can prove that if $(X,r)$ is right non-degenerate,
then $\im(\rho')\subseteq\Aut(A'(X,r))$.

\end{proof}

\begin{proposition}\label{cocycle}
Let $(X,r)$ be a set-theoretic solution of the YBE. Then
\begin{itemize}
\item[(i)] there is a unique $1$-cocycle $\pi:
M(X,r)\longrightarrow A(X,r)$ with respect to the left action
$\lambda'$ such that $\pi(x)=x$, for all $x\in X$.
\item[(ii)] there is a unique $1$-cocycle $\pi':
M(X,r)\longrightarrow A'(X,r)$ with respect to the right action $\rho'$
such that $\pi'(x)=x$, for all $x\in X$.
\end{itemize}
Furthermore, the mapping
 $$f: M(X,r) \rightarrow A(X,r) \rtimes \im(\lambda'):
  a\mapsto (\pi (a) , \lambda_a')$$
is a monoid homomorphism and the mapping
 $$f': M(X,r) \rightarrow A'(X,r)^{op} \rtimes \im(\rho'):
  a\mapsto (\pi'(a),\rho'_a) $$
is a monoid anti-homomorphism.
\end{proposition}

\begin{proof}
We define for $x_1,\dots ,x_m \in X$,
\begin{align*}
\pi(1) &= 0,
\\ \pi(x_1) &=x_1,\quad\mbox{ and for }m>1,
\\ \pi(x_1\circ\cdots\circ x_m) &= x_1 + \lambda'_{x_1}(\pi(x_2\circ\cdots\circ x_m)),
\\ \pi'(1) &= 0',
\\ \pi'(x_1) &=x_1,\quad\mbox{ and for }m>1,
\\ \pi'(x_1\circ\cdots \circ x_m) &= \rho'_{x_m}(\pi'(x_1 \circ \cdots \circ x_{m-1})) \oplus x_m.
\end{align*}
We prove that $\pi(x_1\circ\cdots\circ x_m)$ and
$\pi'(x_1\circ\cdots\circ x_m)$ are well-defined by induction on
$m$. For $m=1$ it is clear. Suppose that $m>1$ and that
$\pi(x_1\circ\cdots\circ x_{m-1})$ and $\pi'(x_1\circ\cdots\circ
x_{m-1})$ are well-defined.

By the induction hypothesis, it is enough to show that
\begin{equation}\label{eq1}
x_1 + \lambda'_{x_1}(\pi(x_2\circ\cdots\circ x_m))=\sigma_{x_1}(x_2)
+ \lambda'_{\sigma_{x_1}(x_2)}(\pi(\gamma_{x_2}(x_1)\circ
x_3\circ\cdots\circ x_m))
\end{equation}
and
\begin{equation}\label{eq2}
\rho'_{x_m}(\pi'(x_1\circ\cdots\circ x_{m-1}))\oplus x_m=
\rho'_{\gamma_{x_m}(x_{m-1})}(\pi'(x_1\circ\cdots\circ x_{m-2}\circ
\sigma_{x_{m-1}}(x_m)))\oplus \gamma_{x_m}(x_{m-1}).
\end{equation}
By (\ref{eq:lambda'1}) and (\ref{eq:rho'1})  we get that
\begin{eqnarray*}
\lefteqn{\sigma_{x_1}(x_2) +
\lambda'_{\sigma_{x_1}(x_2)}(\pi(\gamma_{x_2}(x_1)\circ
x_3\circ\cdots\circ x_m))}\\
&=&\sigma_{x_1}(x_2) +
\lambda'_{\sigma_{x_1}(x_2)}(\gamma_{x_2}(x_1)+\lambda'_{\gamma_{x_2}(x_1)}(\pi(
x_3\circ\cdots\circ x_m)))\\
&=&\sigma_{x_1}(x_2) +
\sigma_{\sigma_{x_1}(x_2)}(\gamma_{x_2}(x_1))+\lambda'_{\sigma_{x_1}(x_2)}(\lambda'_{\gamma_{x_2}(x_1)}(\pi(
x_3\circ\cdots\circ x_m)))\\
&=&x_1+
\sigma_{x_1}(x_2)+\lambda'_{\sigma_{x_1}(x_2)\circ\gamma_{x_2}(x_1)}(\pi(
x_3\circ\cdots\circ x_m))\\
&=&x_1+ \sigma_{x_1}(x_2)+\lambda'_{x_1\circ x_2}(\pi(
x_3\circ\cdots\circ x_m))\\
&=&x_1+ \sigma_{x_1}(x_2)+\lambda'_{x_1}(\lambda'_{x_2}(\pi(
x_3\circ\cdots\circ x_m)))\\
&=&x_1+ \lambda'_{x_1}(x_2+\lambda'_{x_2}(\pi(
x_3\circ\cdots\circ x_m)))\\
&=&x_1+ \lambda'_{x_1}(\pi(x_2\circ\cdots\circ x_m))\\
\end{eqnarray*}
and
\begin{eqnarray*}
\lefteqn{\rho'_{\gamma_{x_m}(x_{m-1})}(\pi'(x_1\circ\cdots\circ
x_{m-2}\circ
\sigma_{x_{m-1}}(x_m)))\oplus \gamma_{x_m}(x_{m-1})}\\
&=&\rho'_{\gamma_{x_m}(x_{m-1})}(\rho'_{\sigma_{x_{m-1}}(x_m)}(\pi'(x_1\circ\cdots\circ
x_{m-2}))\oplus \sigma_{x_{m-1}}(x_m))\oplus \gamma_{x_m}(x_{m-1})\\
&=&\rho'_{\gamma_{x_m}(x_{m-1})}(\rho'_{\sigma_{x_{m-1}}(x_m)}(\pi'(x_1\circ\cdots\circ
x_{m-2})))\oplus \gamma_{\gamma_{x_{m}}(x_{m-1})}(\sigma_{x_{m-1}}(x_m))\oplus \gamma_{x_m}(x_{m-1})\\
&=&\rho'_{\sigma_{x_{m-1}}(x_m)\circ\gamma_{x_m}(x_{m-1})}(\pi'(x_1\circ\cdots\circ
x_{m-2}))\oplus \gamma_{x_m}(x_{m-1})\oplus x_m\\
&=&\rho'_{x_{m-1}\circ x_m}(\pi'(x_1\circ\cdots\circ
x_{m-2}))\oplus \gamma_{x_m}(x_{m-1})\oplus x_m\\
&=&\rho'_{x_m}(\rho'_{x_{m-1}}(\pi'(x_1\circ\cdots\circ
x_{m-2})))\oplus \gamma_{x_m}(x_{m-1})\oplus x_m\\
&=&\rho'_{x_m}(\rho'_{x_{m-1}}(\pi'(x_1\circ\cdots\circ
x_{m-2}))\oplus x_{m-1})\oplus x_m\\
&=&\rho'_{x_m}(\pi'(x_1\circ\cdots\circ
x_{m-1}))\oplus x_m.\\
\end{eqnarray*}
Thus, indeed, $\pi$ and $\pi'$ are well-defined.

For all $a,b\in M(X,r)$, we shall prove by induction on $\deg(a)+\deg(b)$ that
\begin{equation}\label{eq3}\pi(a\circ
b)=\pi(a)+\lambda'_a(\pi(b))\end{equation} and
\begin{equation}\label{eq4}
\pi'(a\circ b)=\rho'_b(\pi'(a))\oplus \pi'(b).\end{equation}
If
$\deg(a)=\deg(b)=1$, then (\ref{eq3}) and (\ref{eq4}) follow by
definition. Hence, we may suppose that  $\deg(a)+\deg(b)>2$ and that $\pi(a'\circ
b')=\pi(a')+\lambda'_{a'}(\pi(b'))$ and $\pi'(a'\circ
b')=\rho'_{b'}(\pi'(a'))\oplus \pi'(b')$, for all $a',b'\in M(X,r)$
such that $\deg(a')+\deg(b')<\deg(a)+\deg(b)$.

Write  $a=x\circ a'$ and $b=b'\circ y$ for some $x,y\in X$ and $a',b'\in M(X,r)$.
 By the induction hypothesis we have
\begin{align*}
\pi(a\circ b)&= \pi(x\circ a'\circ b)\\
&= x+\lambda'_x(\pi(a'\circ b))\\
&= x+\lambda'_x(\pi(a')+\lambda'_{a'}(\pi( b)))\\
&= x+\lambda'_x(\pi(a'))+\lambda'_x(\lambda'_{a'}(\pi( b)))\\
&= \pi(x\circ a')+\lambda'_{x\circ a'}(\pi( b))\\
&= \pi(a)+\lambda'_{a}(\pi( b))\\
\end{align*}
and
\begin{align*}
\pi'(a\circ b)&= \pi'(a\circ b'\circ y)\\
&= \rho'_y(\pi'(a\circ b'))\oplus y\\
&= \rho'_y(\rho'_{b'}(\pi'(a))\oplus\pi'(b'))\oplus y\\
&= \rho'_y(\rho'_{b'}(\pi'(a)))\oplus\rho'_y(\pi'(b'))\oplus y\\
&= \rho'_{b'\circ y}(\pi'(a))\oplus\pi'(b'\circ y)\\
&= \rho'_{b}(\pi'(a))\oplus\pi'(b).\\
\end{align*}
Thus (\ref{eq3}) and (\ref{eq4}) follow by induction. It is clear
that $\pi$ and $\pi'$ are the unique $1$-cocycles satisfying the
hypothesis. Therefore the result follows.
\end{proof}

A natural question is the following.
\begin{question}\label{question}
When are  the 1-cocycles $\pi$ and $\pi'$ bijective?
\end{question}

In general, these 1-cocycles are not bijective.  We provide two
examples. The first one is an example where $\pi$ is injective but
not surjective, and the second one where $\pi$ and $\pi'$ are
neither injective nor surjective.

\begin{example}\label{Ex1}{\rm
    Let $(X,r)$ be a set-theoretic solution of the YBE, where $X$ is set of cardinality greater than $1$ and $r:X \times X \rightarrow X \times X$
is a map defined by $r(x,y) = (x,x)$, for all $x,y \in X$. The
associated monoids are
    \begin{align*}
    & M(X,r) = \langle X \mid x \circ y = x \circ x , \text{ for all } x,y \in X \rangle, \\
    & A(X,r) = \langle X \mid x + x = x + x , \text{ for all } x,y \in X \rangle, \\
    & A'(X,r) = \langle X \mid x \oplus y = x \oplus x , \text{ for all } x,y \in X \rangle.
    \end{align*}
The 1-cocycle $\pi'$ is bijective, but it is clear that the
1-cocycle $\pi$ is not. The latter is not surjective, for example
the element $x + y$, where $x \neq y \in X$ is not in the image of
$\pi$. Note that $\pi$ is still injective. Similarly, $(X,r)$ with
$r: X \times X \rightarrow X \times X$ defined by $r(x,y)=(y,y)$ is
an example of a set-theoretic solution of the YBE where $\pi'$ is
injective but not surjective.}
\end{example}

\begin{example} {\rm
    Let $S = \{0,1,2\}$ and define the skew lattice $(S, \land, \lor)$ by
    $$
    \begin{array}{r|rrr}
    \land & 0 & 1 & 2 \\
    \hline
    0 & 0 & 0 & 0 \\
    1 & 0 & 1 & 2 \\
    2 & 0 & 1 & 2
    \end{array} \hspace{.5cm}
    \begin{array}{r|rrr}
    \lor & 0 & 1 & 2\\
    \hline
    0 & 0 & 1 & 2 \\
    1 & 1 & 1 & 1 \\
    2 & 2 & 2 & 2
    \end{array}
    $$
    The skew lattice $(S,\land, \lor)$ is an example of a distributive and cancellative skew lattice that is not a co-strongly
distributive skew lattice, see (\cite[Example~4.4]{CVV}). By
\cite[Theorem~5.7]{CVV}, $(S,\land,\lor)$ is a left distributive
solution, i.e. $(S,r)$ is a set-theoretic solution of the YBE, where
$r:S \times S \rightarrow S \times S$ is defined by $r(x,y) = (x
\land y, y\lor x)$ for all $x,y \in S$. The associated monoids are
    \begin{align*}
    & M(X,r)= \langle 0,1,2 \mid 1 \circ 0 = 0 \circ 1, 2 \circ 0 = 0 \circ 2, 1\circ 2=2\circ 2, 2 \circ 1 = 1 \circ 1 \rangle, \\
    & A(X,r)= \langle 0,1,2 \mid 1 + 0 = 0 + 0, 2 + 0 = 0 +0, 1+2=2+2, 2+1=1+1 \rangle, \\
    & A'(X,r)= \langle 0,1,2 \mid 1 \oplus 0 = 1 \oplus 1, 2 \oplus 0 = 2 \oplus 2 \rangle.
    \end{align*}
    Both $\pi$ and $\pi'$ are not injective, as $\pi(1 \circ 0) = 1 + 0 = 0 +0 = \pi(0 \circ 0)$ and
$\pi'(1 \circ 0) = 1 \oplus 0 = 1 \oplus 1 = \pi'(1 \circ 1)$, but
$1 \circ 0 \neq 0 \circ 0$ and $1 \circ 0 \neq 1 \circ 1$ in
$M(X,r)$. Both $\pi$ and $\pi'$ are not surjective as $0+1$ (resp.
$0 \oplus 1$) is not in the image of $\pi$ (resp. $\pi'$).}
\end{example}

\begin{proposition}\label{surjective}
Let $(X,r)$ be a set-theoretic solution of the YBE. Write
$r(x,y)=(\sigma_x(y),\gamma_y(x))$. Let  $\pi\colon
M(X,r)\rightarrow A(X,r)$ and $\pi'\colon M(X,r)\rightarrow A'(X,r)$
be the $1$-cocycles of Proposition~\ref{cocycle}. Then
\begin{itemize}
\item[(i)]$\pi$ is surjective if and only if $\sigma_x$ is surjective for
all $x\in X$,
\item[(ii)] $\pi'$ is surjective if and only if $\gamma_x$ is surjective for
all $x\in X$.
\end{itemize}
\end{proposition}

\begin{proof}
Suppose that $\sigma_x$ is surjective for all $x\in X$.  First, we
claim that $\sigma_x$ being surjective implies that $\lambda_x'$ is
surjective. Take $n$ an arbitrary positive integer. Let $x_1, \dots
, x_n \in X$ such that $x_1 + \cdots + x_n \in A(X,r)$. As
$\sigma_x$ is surjective, there exist $y_1, \dots , y_n \in X$ such
that $\sigma_x(y_i)=x_i$, for all $i \in \{1, \dots, n\}$. Then,
$\lambda_x'(y_1 + \cdots + y_n) = \sigma_x(y_1) + \cdots +
\sigma_x(y_n) = x_1 + \cdots + x_n$, which proves that $\lambda_x'$
is surjective.

Next, we prove that $\pi$ is surjective by induction on the length of
the elements in $A(X,r)$. As $\pi(x)=x$ for all $x \in X$, $\pi$ is
surjective on elements of length 1.  Assume now that for a fixed
positive integer $n$ and for any $x_1,\dots ,x_n \in X$, there exist
$y_1,\dots ,y_n \in X$ such that $\pi(y_1 \circ \cdots \circ y_n) =
x_1 + \cdots + x_n$. Take $x_1, \dots , x_{n+1} \in X$. Since
$\lambda_{x_1}'$ is surjective, there exists $z_2,\dots , z_{n+1}
\in X$ such that $\lambda_{x_1}'(z_2 + \cdots + z_{n+1}) = x_2 +
\cdots + x_{n+1}$. Using the induction hypotheses, there exists
$y_2,\dots , y_{n+1} \in X$ such that $\pi(y_2 \circ \cdots \circ
y_{n+1}) = z_2 + \cdots + z_{n+1}$. Thus, we obtain
    \begin{align*}
    x_1+ \cdots + x_{n+1}
    & = x_1 + \lambda_{x_1}'(z_2 + \cdots + z_{n+1})
    \\& = x_1 + \lambda_{x_1}'(\pi(y_2 \circ \cdots \circ y_{n+1}))
    \\& = \pi(x_1 \circ y_2 \circ \cdots \circ y_{n+1}),
    \end{align*}
    and $\pi$ is surjective.

Suppose now that $\pi$ is surjective. Let $x,y\in X$  and consider $x+y \in A(X,r)$. Since $\pi$ is
surjective (and it preserves  the degree), there exist $z,t\in X$
such that $\pi(z\circ t)=x+y$. Thus $z+\sigma_z(t)=x+y$  in $A(X,r)$. By the
defining relations of $A(X,r)$, this equality implies that there
exists $y'\in X$ such that $\sigma_x(y')=y$. Therefore $\sigma_x$ is
surjective for all $x\in X$.

The proof for $\pi'$ is similar.
\end{proof}

\begin{proposition}\label{injective}
Let $(X,r)$ be a set-theoretic solution of the YBE.
Write $r(x,y)=(\sigma_x(y),\gamma_y(x))$. Let $\pi\colon M(X,r)\rightarrow
A(X,r)$ and $\pi'\colon M(X,r)\rightarrow A'(X,r)$ be the
$1$-cocycles of Proposition~\ref{cocycle}.
\begin{itemize}
\item[(i)] If $\sigma_x$ is injective for
all $x\in X$, then $\pi$ is injective.
\item[(ii)] If $\gamma_x$ is injective for
all $x\in X$, then $\pi'$ is injective.
\end{itemize}
\end{proposition}

\begin{proof}
We shall prove $(i)$. The proof of $(ii)$ is similar. Let $\FM(X)$
be the (multiplicative) free monoid on $X$. Suppose that $\sigma_x$
is injective for all $x\in X$. Since $\pi(x)=x$ for all $x\in X$,
the restriction of $\pi$ to elements of degree one in $M(X,r)$ is
injective. Let $n$ be an integer greater than $1$. Let $x_1,\dots,
x_n,y_1,\dots, y_n\in X$ be elements such that
$\pi(x_1\circ\cdots\circ x_n)=\pi(y_1\circ\cdots\circ y_n)$. Thus,
in $A(X,r)$, we have that
$$x_1+\sigma_{x_1}(x_2)+\cdots +\sigma_{x_1}\cdots\sigma_{x_{n-1}}(x_n)
=y_1+\sigma_{y_1}(y_2)+\cdots
+\sigma_{y_1}\cdots\sigma_{y_{n-1}}(y_n).$$ Let $w_1,w_2\in \FM(X)$
be two elements of degree $n$. Suppose that $w_1=z_1\cdots z_n$ and
$w_2=t_1\cdots t_n$, for some $z_i,t_i\in X$. We say that $w_1\sim
w_2$ if there exist $1\leq i\leq n-1$ and $z\in X$ such that
$z_j=t_j$, for all $j\in\{ 1,2,\dots ,n\}\setminus \{ i,i+1\}$ and
either $z_{i+1}=\sigma_{z_{i}}(z)=t_i$ and
$t_{i+1}=\sigma_{t_i}\gamma_{z}(z_i)$, or
$t_{i+1}=\sigma_{t_{i}}(z)=z_i$ and
$z_{i+1}=\sigma_{z_i}\gamma_{z}(t_i)$. Note that $z_1+\dots
+z_n=t_1+\cdots +t_n$ in $A(X,r)$ if and only if there exist
$w_1',\dots ,w_m'\in \FM(X)$ of degree $n$ such that
$$w_1=w_1'\sim w_2'\sim\cdots\sim w_m'=w_2.$$
Hence to prove that $x_1\circ\cdots\circ x_n=y_1\circ\cdots\circ
y_n$, we may assume that there exist $1\leq i\leq n-1$ and $z\in X$
such that
$\sigma_{x_1}\cdots\sigma_{x_{j-1}}(x_j)=\sigma_{y_1}\cdots\sigma_{y_{j-1}}(y_j)$,
for all $j\in\{ 1,2,\dots ,n\}\setminus \{ i,i+1\}$, and also
$\sigma_{x_1}\cdots\sigma_{x_i}(x_{i+1})=\sigma_{\sigma_{x_1}\cdots\sigma_{x_{i-1}}(x_{i})}(z)
=\sigma_{y_1}\cdots\sigma_{y_{i-1}}(y_{i})$ as well as
 $\sigma_{y_1}\cdots\sigma_{y_i}(y_{i+1})=\sigma_{\sigma_{y_1}\cdots\sigma_{y_{i-1}}(y_{i})}\gamma_{z}(\sigma_{x_1}\cdots\sigma_{x_{i-1}}(x_i))$.

Since
$\sigma_{x_1}\cdots\sigma_{x_{j-1}}(x_j)=\sigma_{y_1}\cdots\sigma_{y_{j-1}}(y_j)$,
for all $j\in\{ 1,2,\dots ,n\}\setminus \{ i,i+1\}$, and $\sigma_x$
is injective for all $x\in X$, we have that $x_j=y_j$, for all $j\in
\{ 1,\dots, i-1\}$.  Hence, since
$\sigma_{x_1}\cdots\sigma_{x_i}(x_{i+1})=\sigma_{y_1}\cdots\sigma_{y_{i-1}}(y_{i})$,
and  $\sigma_x$ is injective for all $x\in X$, we have that
$y_i=\sigma_{x_i}(x_{i+1})$. Now we have that
\begin{eqnarray*}
\sigma_{\sigma_{x_1}\cdots\sigma_{x_{i-1}}(x_{i})}(z)&=&\sigma_{x_1}\cdots\sigma_{x_i}(x_{i+1})\\
&=&\lambda_{x_{1}\circ\cdots\circ x_{i-1}}\lambda_{x_i}(x_{i+1})\\
&=&\lambda_{\lambda_{x_{1}\circ\cdots\circ
x_{i-1}}(x_i)}\lambda_{\rho_{x_i}(x_{1}\circ\cdots\circ
x_{i-1})}(x_{i+1})\\
&=&\sigma_{\sigma_{x_{1}}\cdots\sigma_{x_{i-1}}(x_i)}\lambda_{\rho_{x_i}(x_{1}\circ\cdots\circ
x_{i-1})}(x_{i+1}),
\end{eqnarray*}
where the third equality follows by Theorem~\ref{main}.

Hence, since $\sigma_x$ is injective for all $x\in X$, we have that
$$z=\lambda_{\rho_{x_i}(x_{1}\circ\cdots\circ
x_{i-1})}(x_{i+1}).$$ By Theorem~\ref{main},
\begin{eqnarray*}
\lefteqn{\sigma_{y_1}\cdots\sigma_{y_i}(y_{i+1})}\\
&=&\sigma_{\sigma_{y_1}\cdots\sigma_{y_{i-1}}(y_{i})}\gamma_{z}(\sigma_{x_1}\cdots\sigma_{x_{i-1}}(x_i))\\
&=&\sigma_{\sigma_{x_1}\cdots\sigma_{x_{i-1}}(\sigma_{x_i}(x_{i+1}))}\gamma_{z}(\sigma_{x_1}\cdots\sigma_{x_{i-1}}(x_i))\\
&=&\lambda_{\lambda_{x_1\circ\cdots\circ
x_{i-1}}(\lambda_{x_i}(x_{i+1}))}\rho_{\lambda_{\rho_{x_i}(x_{1}\circ\cdots\circ
x_{i-1})}(x_{i+1})}(\lambda_{x_1\circ\cdots\circ x_{i-1}}(x_i))\\
&=&\lambda_{\lambda_{x_1\circ\cdots\circ
x_{i-1}}(\lambda_{x_i}(x_{i+1}))}\lambda_{\rho_{\lambda_{x_i}(x_{i+1})}(x_{1}\circ\cdots\circ
x_{i-1})}(\rho_{x_{i+1}}(x_i))\\
&=&\lambda_{x_1\circ\cdots\circ
x_{i-1}}\lambda_{\lambda_{x_i}(x_{i+1})}(\rho_{x_{i+1}}(x_i))\\
&=&\lambda_{y_1\circ\cdots\circ
y_{i-1}}\lambda_{y_i}(\rho_{x_{i+1}}(x_i))\\
&=&\sigma_{y_1}\cdots\sigma_{y_{i-1}}\sigma_{y_i}(\gamma_{x_{i+1}}(x_i)).
\end{eqnarray*}
Since $\sigma_x$ is injective for all $x\in X$, we have that
$y_{i+1}=\gamma_{x_{i+1}}(x_i)$. Thus,
$$y_i\circ y_{i+1}=\sigma_{x_i}(x_{i+1})\circ \gamma_{x_{i+1}}(x_i)=x_i\circ x_{i+1}.$$
Since
$\sigma_{x_1}\cdots\sigma_{x_{j-1}}(x_j)=\sigma_{y_1}\cdots\sigma_{y_{j-1}}(y_j)$,
for all $j\in\{ 1,2,\dots ,n\}\setminus \{ i,i+1\}$, and $\sigma_x$
is injective for all $x\in X$, we have that $x_j=y_j$, for all $j\in
\{ i+2,\dots, n\}$. Hence $x_1\circ\cdots\circ
x_n=y_1\circ\cdots\circ y_n$, and therefore, $\pi$ is injective.
\end{proof}

\begin{remark} {\rm Note that in the set-theoretic solution of the
YBE of Example \ref{Ex1}, $\sigma_x(y)=x$ for all $x,y \in X$, so
$\sigma_x$ is not injective. But, $\pi$ is injective.
Similarly, $(X,r)$ with $r: X \times X \rightarrow X \times X$
defined by $r(x,y)=(y,y)$ is a set-theoretic solution of the YBE
where $\pi'$ is injective (see Example \ref{Ex1}) but
$\gamma_y(x)=y$ for all $x,y \in X$. So $\gamma_y$ is not
injective.}
\end{remark}

If $\pi$ (resp. $\pi'$) is injective, then it is clear that the map
$f$ (resp. $f'$), defined in Proposition~\ref{cocycle}, is an
embedding. The latter was proved in \cite{JKV} under the assumption
that $(X,r)$ is a left non-degenerate solution. In this case $\pi$
is bijective and $M(X,r)$ is a regular submonoid of the semidirect
product $A(X,r) \rtimes \gr(\sigma_x \mid x \in X)$.

The following result answers Question~\ref{question} for
finite solutions.

\begin{corollary}\cite{JKV} (Jespers, Kubat and Van Antwerpen)\label{bijectivecocycle}
Let $(X,r)$ be a set-theoretic solution of the YBE, $\lambda'$
(resp. $\rho'$) the left (resp. right) action as defined before,
$\pi$ (resp. $\pi'$) the unique 1-cocycle with respect to $\lambda'$
(resp. $\rho'$). Then, $\pi$ (resp. $\pi'$) is bijective if $(X,r)$
is left non-degenerate (resp. right non-degenerate). The converse
holds if $X$ is finite.
\end{corollary}
\begin{proof}
Assume first that $(X,r)$ is a left non-degenerate set-theoretic
solution of the YBE. Then, by Propositions~\ref{surjective}
and~\ref{injective}, $\pi$ is bijective. Similarly, one can prove
that $(X,r)$ being a right non-degenerate solution implies that
$\pi'$ is bijective.

Assume now that $\pi: M(X,r) \rightarrow A(X,r)$ is bijective and $X$ is finite. By
Proposition~\ref{surjective}, $\sigma_x$ is surjective for all $x
\in X$. Since $X$ is finite, $\sigma_x$ is bijective for all $x \in
X$, that is $(X,r)$ is left non-degenerate.
\end{proof}

The next example shows the difficulty of Question~\ref{question} for
infinite solutions.

\begin{example}
Consider the set $\mathbb{N}$ of the non-negative integers. Let
$r\colon \mathbb{N}\times \mathbb{N}\longrightarrow\mathbb{N}\times
\mathbb{N}$ be the map defined by $r(x,y)=(\xi(y),\xi(x))$, for all
$x,y\in\mathbb{N}$, where $\xi(x)=\max\{ 0, x-1\}$, for all $x\in
\mathbb{N}$. Then $(\mathbb{N}, r)$ is a set-theoretic solution of
the YBE, such that the associated $1$-cocycles $\pi$ and $\pi'$ are
bijective, but, for every $x\in\mathbb{N}$ $\sigma_x=\gamma_x=\xi$
is not injective, because $\xi(0)= \xi(1)$.
\end{example}
\begin{proof}
It is easy to check that $(\mathbb{N},r)$ is a set-theoretic
solution of the YBE. Note that, for every $x\in\mathbb{N}$,
$\xi^x(x)=0$. Hence
$$M(\mathbb{N},r)=\langle\mathbb{N}\mid x\circ y=0\circ 0\rangle,$$
$$A(\mathbb{N},r)=\langle\mathbb{N}\mid x+ y=0+ 0\rangle$$
and
$$A'(\mathbb{N},r)=\langle\mathbb{N}\mid x\oplus y=0\oplus 0\rangle.$$
Therefore, for every integer $n>1$, the monoids $M(\mathbb{N},r)$,
$A(\mathbb{N},r)$ and $A'(\mathbb{N},r)$ have only one element of
degree $n$. Since $\pi$ and $\pi'$ preserve the degree and
$\pi(x)=x$ and $\pi'(x)=x$, for all $x\in \mathbb{N}$, we have that
$\pi$ and $\pi'$ are bijective. Thus the result follows.
\end{proof}

\section{Non-degenerate irretractable solutions}

In \cite[Theorem 2]{RumpAdv} (and independently in \cite[Corollary 2.3]{JO}), it is proven that any finite involutive left
non-degenerate set-theoretic solution of the YBE also is right
non-degenerate. In the infinite case, the latter is no longer true.
The following example from \cite{RumpAdv} shows this.

\begin{example}\rm \label{exref}
Let X be the set of the integers, and define
$r: X^2\rightarrow X^2$ by $$r(x,y)=(\lambda_x(y),\lambda^{-1}_{\lambda_x(y)}(x)),$$
where $\lambda_x(y)=y+\min(x,0)$, for all $x,y\in X$. Note that $\lambda_x$ is bijective
and $\lambda^{-1}_x(y)=y-\min(x,0)$. It is easy to check that $(X,r)$
is an involutive solution. Note that it is not right non-degenerate. In fact, if $a<0$, we have that
\begin{align*}
    \rho_a(b)=\lambda^{-1}_{\lambda_{b}(a)}(b)=b-\min(a+\min(b,0),0)=b-(a+b)=-a,
\end{align*}
for all $b<0$. Hence $\rho_a$ is not bijective if $a<0$.

\end{example}

It is unclear whether the above holds for aribtrary bijective
solutions. Hence the following question is pertinent.

\begin{question}
Is any finite bijective left non-degenerate set-theoretic solution
of the YBE right non-degenerate?
\end{question}

A natural question is the converse:

\begin{question}
Are non-degenerate solutions of the YBE always bijective?
\end{question}

We will give a positive answer to this question in case the solution $(X,r)$ is irretractable,
i.e.  $\sigma_x=\sigma_y$ implies  $x=y$, for all $x,y\in X$. Note that
Example~\ref{exref} is a retractable involutive solution. To our knowledge it is unknown
whether there exist infinite involutive irretractable solutions that are left but not
right non-degenerate.
Note that irretractability has been defined with respect to to the maps $\sigma_x$. One could equally well define retractabilty with respect to the maps $\gamma_x$. However, this makes no difference since any solution
$r(x,y)=(\sigma_x(y),\gamma_y(x))$ has a dual solution $r'(y,x)=(\gamma_y(x),\sigma_x(y))$. Clearly $r$ is (bijective) non-degenerate if and only if $r'$ is (bijective) non-degenerate.

To prove the result we will make use of  the following result of Rump  \cite[Proposition 1]{R}:
Let $X$ be a non-empty set and let $r:X\times X\rightarrow X\times X$ be a map, with $r(x,y)=(\sigma_x(y),\gamma_y(x))$, such that $\gamma_y :X\rightarrow X$ is bijective for all $y\in X$. Then $(X,r)$ is a solution of the YBE if and only if the following conditions hold for all $x,y,z \in X$:
\begin{enumerate}
\item[(R1)]  $ (x \cdot y)\cdot (x\cdot z)=(y:x)\cdot (y\cdot z)$,
\item[(R2)]  $ (x:y):(x:z)=(y\cdot x):(y:z)$,
\item[(R3)]  $ (x\cdot y):(x\cdot z)=(y:x)\cdot (y:z)$,
\end{enumerate}
where $x\cdot y=\gamma^{-1}_x(y)$,  $x:y=\sigma_{\gamma^{-1}_y(x)}(y)$.
Furthermore, $r$ is a bijective solution if the map $X\rightarrow X$ defined by $z\mapsto x:z$ is bijective.
The use of this result has been proposed by the referee to avoid the arboresque sub- and superscripts in the original proof.

We also will make use of a  lemma  that was proved by Lebed and Vendramin in \cite{LV} for
finite non-degenerate bijective solutions.

\begin{lemma}\label{h}
Let $(X,r)$  be a non-degenerate  set-theoretic solution of the
YBE. Let $h\colon X\rightarrow X$ be the map defined by
$h(x)=\sigma^{-1}_x(x)$, for all $x\in X$. If $(X,r)$ is
irretractable,
then $h$ is bijective and $h^{-1}(x)=\gamma^{-1}_x(x)$, for all
$x\in X$.
\end{lemma}

\begin{proof}
As commented above, we may assume that $(X,r)$ is a  non-degenerate set-theoretic solution of the YBE such that
$\gamma_x=\gamma_y$ implies that $x=y$.
Thus  conditions (R1),(R2),(R3) hold. Then, by (R1),
  $$\gamma^{-1}_{x\cdot x}(x\cdot z)=\gamma^{-1}_{x:x}(x\cdot z),$$
  for all $x,z\in X$. Hence,
      $$x\cdot x=x:x.$$
 Now $x:x=\sigma_{x\cdot x}(x)$ and thus $\sigma^{-1}_{x\cdot x}(x\cdot x)=x$. This shows that the map
  $x\mapsto x\cdot x =\gamma^{-1}_x(x)$ is injective.
For $x,y\in X$, put
  $$\sigma^{-1}_x(y)=x*y.$$ For $y=x*x$, we have
  $x=\sigma_x(y)=\sigma_{\gamma^{-1}_y(\gamma_y(x))}(y)=\gamma_y(x) : y$.
Hence by (R1)
 $$x\cdot (\gamma_y(x)\cdot z)=(\gamma_y(x) :y)\cdot (\gamma_y(x)\cdot z) =(y\cdot \gamma_y(x))\cdot (y\cdot z)
=x\cdot (y\cdot z).$$
Therefore $\gamma_y(x)\cdot z=y\cdot z$, which yields
$\gamma_y(x)=y$. Hence
   \begin{eqnarray}  \label{dotstar}
   x&=&y\cdot y=(x*x)\cdot (x*x),
   \end{eqnarray}
 which shows that the map $x\mapsto x\cdot x=\gamma^{-1}_x(x)$ is bijective.  Furthermore the inverse of this map is the map
  $
   x\mapsto x*x= \sigma_x^{-1}(x) =h(x)
   $.
\end{proof}

\begin{theorem}\label{MainIrr}
Let $(X,r)$ be an irretractable non-degenerate set-theoretic solution of the YBE. Then $r$ is bijective.
\end{theorem}
\begin{proof}
Again we may assume that $(X,r)$ is a  non-degenerate set-theoretic solution of the YBE such that
$\gamma_x=\gamma_y$ implies that $x=y$.
From (R3) we get that
  $$x:(z\cdot z) =(z\cdot\gamma_{z}(x)): (z\cdot z) =(\gamma_z (x) :z)\cdot (\gamma_{z}(x):z) =\sigma_x(z) \cdot \sigma_x(z).$$
From Lemma~\ref{h} we then get that (see equation (\ref{dotstar}))
  $$x:z =\sigma_{x}(z*z) \cdot \sigma_x (z*z).$$
  Since $x*x=\sigma_x^{-1}(x) = h(x)$, we get from Lemma~\ref{h} that
   the map $z\mapsto x : z =\sigma_{\gamma^{-1}_{z}(x)}(z)$ is bijective.
Hence, by Rump's earlier mentioned result,   $r$ is bijective.
\end{proof}

Note that if $(X,r)$ is an irretractable non-degenerate solution
then for every  $x \in X$ there is a  unique  $y \in X$ such that
$r(x,y) = (x,y)$ and there is a unique $z\in X$ such that
$r(z,x)=(z,x)$. Because $(X,r)$ is left non-degenerate, to prove the
former, it is sufficient to show that $\sigma_x (y)=x$ implies
$\gamma_y(x)=y$. Now, because $(X,r)$ is a solution we obtain from
(\ref{YBE1}) that $\sigma_x \sigma_y = \sigma_{\sigma_x(y)}
\sigma_{\gamma_y(x)} = \sigma_x \sigma_{\gamma_y(x)}$ and thus  $
\sigma_y = \sigma_{\gamma_y(x)}$. The irretractble assumption yields
that  $y = \gamma_y(x)$, as claimed. Similarly one proves the other
claim. Hence the there are at least  ${|X| \choose 2}$ defining
relations for the structure monoid. Furthermore, there are precisely
$|X| \choose 2$ defining relations  if $r$ also is involutive and
thus, in this case, $M(X,r)$ is a monoid with a presentation of the
type  $\langle x_1 , \ldots , x_n \mid R\rangle$, where $R$ is a set
consisting of $n\choose 2$ relations of the type $x_i x_j = x_k x_l$
with $(x_i , x_j) \neq (x_k ,x_l)$ and every word $x_i x_j$ appears
in at most one relation. Note that such a presentation has
associated a map $r\colon X\times X\rightarrow X\times X$, where
$X=\{ x_1,\dots x_n\}$, $r^2=\id_{X^2}$ and $r(x_i,x_j)=(x_k,x_l)$
if and only if either $x_ix_j=x_kx_l$ is one of the relations in $R$
or $x_ix_j$ does not appear in any relation in $R$ and
$(x_k,x_l)=(x_i,x_j)$ in this case. Monoids with this type of
presentation and their algebras have a rich algebraic structure,
when $r$ is non-degenerate, even if $(X,r)$ is not a solution of the
YBE. Such   monoids are said to be of quadratic type, and if
$x_ix_i$ does not appear in any defining relation, then they are
said to be of skew type. We refer the reader to
\cite{CO,GJO,JVC1,JVC2}. In \cite{JVC1} it has been shown that for
such a monoid $r$ is a non-degenerate solution of the YBE if and
only if the monoid is cancellative and $r$ is non-degenerate and
satisfies the cyclic condition, i.e. if for every $x_1,y\in X$ there
exist $x_2, y_1,y_2, z_1,z_2\in X$ such that $x_1 y=y_1 z_1$ and
$x_2 y_1=y_2 z_2$ with $r(x_2,x_1)=(x_2,x_1) $ and
$r(z_2,z_1)=(z_2,z_1)$.
 The latter monoids were first investigated by Gateva-Ivanova and Van den Bergh in \cite{GIV}.

\section{The structure left semi-truss}

Braces and skew braces were introduced to deal with bijective
non-degenerate solutions $(X,r)$ of the YBE. In order to translate
such solutions to associative structures the structure group
$G(X,r)$ and the structure monoid $M(X,r)$ were introduced. The
group $G(X,r)$ turns out to be a skew brace, however a structure
monoid does not fit in this context. Recently, Brzezi\'nski
introduced the algebraic notion of a semi-truss which is built on
two semigroup structures on a given set. We show that structure
monoids of left non-degenerate solutions of the YBE fit in this
context: they turn out to be left semi-trusses with additive
structure that is close to being a normal monoid. We then show that
also the least left cancellative epimorphic image of $M(X,r)$
inherits a left non-degenerate solution of the YBE that restricts to
the original solution $r$ for some interesting class, in particular
if $(X,r)$ is irretractable.

We first recall the definition of a left semi-truss.

\begin{definition}\cite{Br}
A left semi-truss is a quadruple $(A,+,\circ,\phi)$ such that
$(A,+)$ and $(A,\circ)$ are non-empty semigroups and $\phi\colon
A\times A\longrightarrow A$ is a function such that
$$a\circ (b+ c)= (a\circ b)+ \phi(a,c),$$
for all $a,b,c\in A$.
\end{definition}

\begin{example}\label{ex1}
{\rm Let $(X,r)$ be a left non-degenerate set-theoretic solution of
the YBE (not necessarily bijective).
 As stated in Section~\ref{derived},
and with the same notation,  the map
$r'(x,y)=(y,\sigma_y\gamma_{\sigma^{-1}_x(y)}(x))$ defines the left
derived solution on $X$. Let $M=M(X,r)$ and $M'=A(X,r)=M(X,r')$ be
the structure monoids of the solutions $(X,r)$ and $(X,r')$
respectively. From Corollary~\ref{bijectivecocycle} and
Proposition~\ref{lambda'rho'} we obtain a left action
$\lambda'\colon (M,\circ )\longrightarrow \Aut(M',+)$ and a
bijective $1$-cocycle $\pi\colon M\longrightarrow M'$ with respect
to $\lambda'$ satisfying $\lambda'(x)(y)=\sigma_x(y)$ and
$\pi(x)=x$, for all $x,y\in X$. We identify $M$ and $M'$ via $\pi$,
that is $a=\pi(a)$ for all $a\in M$. With this identification, we
obtain  the operation $+$ on $M$, and $a\circ b=a+\lambda'_a(b)$,
for all $a,b\in M$. Put $\phi(a,b)=\lambda'_a(b)$, for all $a,b\in
M$. Then
$$a\circ (b+c)=a+\lambda'_a(b+c)=a+\lambda'_a(b)+\lambda'_a(c)=(a\circ b)+\phi(a,c).$$
Furthermore $M + a \subseteq a+M$, for all $a\in M$. Hence
$(M,+,\circ,\phi)$ is a left semi-truss. Note that if, furthermore,
$r$ is bijective then it easily can be verified that $(X,r')$ is a
right non-degenerate solution and thus  $M+a =a+M$ for all $a\in M$;
that is $(M,+)$ consists of normal elements. As shown in \cite{JKV},
this property is fundamental in the study of the associated
structure algebra $KM(X,r)$.}
\end{example}

 In the remainder of this section we show that if $(M,+,\circ,\phi)$ is a left semi-truss such that for every $a,b\in M$ there exists a unique $c(a,b)\in M$ such that $a+b=b+c(a,b)$,
then there exists a set-theoretical solution of the YBE on $M$, say
$(M,r')$. In case that $M=M(X,r)/\eta$, the least cancellative
epimorphic image of $M(X,r)$ it follows that $r'$ is the (unique)
extension of $r$ to $M$.

\begin{lemma}\label{st}
Let $(A,+)$ be a non-empty semigroup such that for each $(a,b)\in A\times A$
there exists a unique $c(a,b)\in A$ such that
$$a+ b=b+ c(a,b).$$
Then $(A,r')$, where
$$r'(a,b)=(b,c(a,b)),$$
for all $a,b\in A$, is a set-theoretic solution of the YBE.
\end{lemma}
\begin{proof}
Let $(a,b,d)\in A^3$. We have
\begin{eqnarray*}
a+ b+ d &=& b+ c(a,b)+ d\\
&=& b+ d+ c(c(a,b),d)\\
\end{eqnarray*}
and also
\begin{eqnarray*}
a+ b+ d &=& a+ d+ c(b,d)\\
&=& d+ c(a,d)+ c(b,d)\\
&=& d+ c(b,d)+ c(c(a,d),c(b,d))\\
&=& b+ d+ c(c(a,d),c(b,d)).
\end{eqnarray*}
Hence, by the uniqueness assumption,
\begin{equation}\label{c}
c(a,b+ d)=c(c(a,b),d)=c(c(a,d),c(b,d)). \end{equation}
Now we have
$$r'_1r'_2r'_1(a,b,d)=r'_1r'_2(b,c(a,b),d)=r'_1(b,d,c(c(a,b),d))=(d,c(b,d),c(c(a,b),d))$$
and
\begin{eqnarray*}r'_2r'_1r'_2(a,b,d)&=&r'_2r'_1(a,d,c(b,d))=r'_2(d,c(a,d),c(b,d))\\
&=&(d,c(b,d),c(c(a,d),c(b,d))).\end{eqnarray*}
Therefore, by
(\ref{c}) $r'_1r'_2r'_1=r'_2r'_1r'_2$, and the result follows.
\end{proof}

\begin{proposition}\label{st2}
Let $(A,+)$ and $(A,\circ)$ be non-empty semigroups. Let $\lambda\colon (A,\circ)\rightarrow \Aut(A,+)$ be a homomorphism such that
$a\circ b=a+\lambda_a(b)$, for all $a,b\in A$, where $\lambda(a)=\lambda_a$. In particular, $(A,+,\circ,\phi)$ is a left semi-truss with
 $\phi (a,b)=\lambda_a(b)$, for all $a,b\in
A$.
Suppose that for each $(a,b)\in A\times A$ there exists
a unique $c(a,b)\in A$ such that
$$a+ b=b+ c(a,b).$$
Then $(A,r)$, where
$$r(a,b)=(\lambda_a(b),\lambda^{-1}_{\lambda_a(b)}(c(a,\lambda_a(b)))),$$
for all $a,b\in A$, is a left non-degenerate set-theoretic solution
of the YBE.
\end{proposition}
\begin{proof} Let $J\colon A^3\longrightarrow A^3$ be the map
defined by $J(a,b,d)=(a,\lambda_a(b),\lambda_a\lambda_b(d))$.
Clearly $J$ is bijective and
$J^{-1}(a,b,d)=(a,\lambda^{-1}_a(b),\lambda^{-1}_{\lambda^{-1}_a(b)}\lambda^{-1}_a(d))$,
for all $a,b,d\in A$. We have
\begin{eqnarray*}
J^{-1}r'_1J(a,b,d)&=&J^{-1}r'_1(a,\lambda_a(b),\lambda_a\lambda_b(d))\\
&=&J^{-1}(\lambda_a(b),c(a,\lambda_a(b)),\lambda_a\lambda_b(d))\\
&=&(\lambda_a(b),\lambda^{-1}_{\lambda_a(b)}(c(a,\lambda_a(b))),\lambda^{-1}_{\lambda^{-1}_{\lambda_a(b)}(c(a,\lambda_a(b)))}\lambda^{-1}_{\lambda_a(b)}\lambda_a\lambda_b(d)),
\end{eqnarray*}
where $r'$ is defined as in Lemma \ref{st}.
Since $a\circ b=a+ \lambda_a(b)=\lambda_a(b)+
c(a,\lambda_a(b))=\lambda_a(b)\circ\lambda^{-1}_{\lambda_a(b)}(c(a,\lambda_a(b)))$
it follows that
 $J^{-1}r'_1J=r_1$. Similarly
\begin{eqnarray*}
J^{-1}r'_2J(a,b,d)&=&J^{-1}r'_2(a,\lambda_a(b),\lambda_a\lambda_b(d))\\
&=&J^{-1}(a,\lambda_a\lambda_b(d),c(\lambda_a(b),\lambda_a\lambda_b(d)))\\
&=&(a,\lambda_b(d),\lambda^{-1}_{\lambda_b(d)}\lambda^{-1}_a(c(\lambda_a(b),\lambda_a\lambda_b(d))))
\end{eqnarray*}
Note that
$$\lambda^{-1}_a(d)+\lambda^{-1}_a(c(b,d))=\lambda^{-1}_a(d+c(b,d))=\lambda^{-1}_a(b+d)=\lambda^{-1}_a(b)+\lambda^{-1}_a(d),$$
for all $a,b,d\in A$. Hence, by the uniqueness assumption,
$\lambda^{-1}_a(c(b,d))=c(\lambda^{-1}_a(b),\lambda^{-1}_a(d))$.
Since each $\lambda_a$ is bijective it follows that
\begin{eqnarray*}
J^{-1}r'_2J(a,b,d)&=&(a,\lambda_b(d),\lambda^{-1}_{\lambda_b(d)}\lambda^{-1}_a(c(\lambda_a(b),\lambda_a\lambda_b(d))))\\
&=&(a,\lambda_b(d),\lambda^{-1}_{\lambda_b(d)}(c(b,\lambda_b(d)))).
\end{eqnarray*}
Thus $J^{-1}r'_2J=r_2$. By Lemma~\ref{st}, $(A,r')$ is a
set-theoretic solution of the YBE. Therefore also $(A,r)$ is a
set-theoretic solution of the YBE, and the result follows.
\end{proof}

Let $(X,r)$ be a left non-degenerate set-theoretic solution of the
YBE. We will write $r(x,y)=(\sigma_x(y),\gamma_y(x))$, for all
$x,y\in X$. Thus the $\sigma_x$ are bijective maps. The derived
solution of $(X,r)$ is $(X,r')$, where
$$r'(x,y)=(y,\sigma_y(\gamma_{\sigma^{-1}_x(y)}(x))),$$
for all $x,y\in X$. We will use the notation of Example~\ref{ex1}.
Thus we have $M=M(X,r)$ and the left semi-truss $(M,+,\circ,\phi)$,
where $\phi(a,b)=\lambda'_a(b)$, for all $a,b\in M$. Recall that
$\lambda'\colon (M,\circ)\longrightarrow \Aut(M,+)$ is an
homomorphism, that is, an action of $(M,\circ)$ on $(M,+)$, and
$\id\colon M\longrightarrow M$ is a bijective $1$-cocycle with
respect to $\lambda'$ (because $a\circ b=a+\lambda'_a(b)$).

Let $\eta$ be the left cancellative congruence on $(M,+)$, that is,
$\eta$ is the smallest congruence such that $\bar M=(M,+)/\eta$ is a
left cancellative monoid.

We shall see a description of the elements in $\eta$. Let
$$\eta_0=\{ (a,b)\in M^2\mid \exists c\in M\mbox{ such that }c+a=c+b\}.$$
Note that $\eta_0$ is a reflexive and symmetric binary relation on
$M$. Let $\eta_1$ be its transitive closure, that is
$$\eta_1=\{(a,b)\in M^2\mid \exists a_1,\dots,a_n\in M \mbox{ such that }(a,a_1),(a_1,a_2),\dots ,(a_n,b)\in \eta_0\}.$$
Thus $\eta_1$ is an equivalence relation on $M$. Let
\begin{align*}\eta_2=&\{ (c+a,c+b)\in M^2\mid c\in M\mbox{ such that
}(a,b)\in \eta_1\}\\
&\cup \{ (a,b)\in M^2\mid \exists c\in M\mbox{ such that
}(c+a,c+b)\in \eta_1\},
\end{align*}
and for every $m\geq 1$ we
define
$$\eta_{2m+1}=\{(a,b)\in M^2\mid \exists a_1,\dots,a_n\in M \mbox{ such that }(a,a_1),(a_1,a_2),\dots ,(a_n,b)\in \eta_{2m}\}$$
and
\begin{align*}\eta_{2m+2}=&\{ (c+a,c+b)\in M^2\mid c\in M\mbox{
such that }(a,b)\in \eta_{2m+1}\}\\
&\cup \{ (a,b)\in M^2\mid \exists c\in M\mbox{ such that
}(c+a,c+b)\in \eta_{2m+1}\},\end{align*}
Note that
$\eta_n\subseteq\eta_{n+1}\subseteq \eta$ for all $n\geq 0$. Let
$\eta'=\cup_{n=0}^{\infty}\eta_n$.

\begin{lemma}\label{lemma1}\label{congruence}
With the above notation we have that $\eta'=\eta$ and
$\lambda'_a=\lambda'_b$, for all $(a,b)\in \eta$. Furthermore, for
all $z\in M$,
$$\eta=\{ ( \lambda'_z(a),\lambda'_z(b))\mid (a,b)\in\eta\} = \{ ( (\lambda'_z)^{-1}(a),(\lambda'_z)^{-1}(b))\mid (a,b)\in\eta\},$$
and $\eta$ also is a congruence on $(M,\circ)$.
\end{lemma}
\begin{proof}
First we shall prove that $\eta'$ is a congruence on $(M,+)$.
Clearly $\eta'$ is reflexive and symmetric because so is each
$\eta_n$. Let $a,b,c\in M$ such that $(a,b),(b,c)\in \eta'$. There
exists a positive integer $m$ such that $(a,b),(b,c)\in
\eta_{2m+1}$. Since $\eta_{2m+1}$ is the transitive closure of
$\eta_{2m}$, we have that $(a,c)\in \eta_{2m+1}\subseteq \eta'$.
Hence $\eta'$ is an equivalence relation. Note that every $\eta_n$
satisfies that $(x+z,y+z)\in\eta_n$,  for all $(x,y)\in \eta_n$.
Hence $(a+c,b+c)\in \eta_{2m+1}\subseteq \eta'$. Since $(a,b)\in
\eta_{2m+1}$, we have that $(c+a,c+b)\in \eta_{2m+2}\subseteq
\eta'$. Therefore $\eta'$ is a congruence.

Let $a,b,c\in M$ be elements such that $(c+a,c+b)\in \eta'$. There
exists a positive integer $t$ such that $(c+a,c+b)\in \eta_{2t+1}$.
Thus $(a,b)\in \eta_{2t+2}\subseteq \eta'$. Hence $(M,+)/\eta'$ is a
left cancellative monoid. Since $\eta'\subseteq \eta$, we have
$\eta'=\eta$ by the definition of $\eta$.

Let $(a,b)\in \eta_0$. Then there exists $c\in M$ such that
$c+a=c+b$. Let $z\in M$. We have that
$$(\lambda'_z)^{\varepsilon}(c)+(\lambda'_z)^{\varepsilon}(a)=(\lambda'_z)^{\varepsilon}(c+a)=(\lambda'_z)^{\varepsilon}(c+b)=(\lambda'_z)^{\varepsilon}(c)+(\lambda'_z)^{\varepsilon}(b),$$
for $\varepsilon=\pm 1$. Therefore
$\eta_0=\{(\lambda'_z(a),\lambda'_z(b))\mid (a,b)\in \eta_0\}= \{ ( (\lambda'_z)^{-1}(a),(\lambda'_z)^{-1}(b))\mid (a,b)\in\eta_0\}$.
Thus, clearly
$$\eta_1=\{(\lambda'_z(a),\lambda'_z(b))\mid (a,b)\in
\eta_1\}= \{ ( (\lambda'_z)^{-1}(a),(\lambda'_z)^{-1}(b))\mid (a,b)\in\eta_1\}.$$ Let $(a,b)\in \eta_2$, then either there exist
$c,a',b'\in M$ such that $(a',b')\in \eta_1$ and
$(a,b)=(c+a',c+b')$, or there exists $d\in M$ such that
$(d+a,d+b)\in \eta_1$.  In the first case, we have that
$$((\lambda'_z)^{\varepsilon}(a),(\lambda'_z)^{\varepsilon}(b))=((\lambda'_z)^{\varepsilon}(c)+(\lambda'_z)^{\varepsilon}(a'),(\lambda'_z)^{\varepsilon}(c)+(\lambda'_z)^{\varepsilon}(b')),$$
for $\varepsilon=\pm 1$. Since
$((\lambda'_z)^{\varepsilon}(a'),(\lambda'_z)^{\varepsilon}(b'))\in
\eta_1$, we get that
$((\lambda'_z)^{\varepsilon}(a),(\lambda'_z)^{\varepsilon}(b))\in
\eta_2$, in this case. In the second case, since $(d+a,d+b)\in
\eta_1$, we have that
$((\lambda'_z)^{\varepsilon}(d)+(\lambda'_z)^{\varepsilon}(a),(\lambda'_z)^{\varepsilon}(d)+(\lambda'_z)^{\varepsilon}(b))\in
\eta_1$. Thus also in this case we have that
$((\lambda'_z)^{\varepsilon}(a),(\lambda'_z)^{\varepsilon}(b))\in
\eta_2$. Therefore
$$\eta_2=\{(\lambda'_z(a),\lambda'_z(b))\mid (a,b)\in
\eta_2\}= \{ ( (\lambda'_z)^{-1}(a),(\lambda'_z)^{-1}(b))\mid (a,b)\in\eta_2\}.$$ Now it is easy to show by induction on $n$ that
$$\eta_n=\{(\lambda'_z(a),\lambda'_z(b))\mid (a,b)\in
\eta_n\}= \{ ( (\lambda'_z)^{-1}(a),(\lambda'_z)^{-1}(b))\mid (a,b)\in\eta_n\},$$ for all non-negative integer $n$. Hence
$$\eta=\{ (\lambda'_z(a),\lambda'_z(b))\mid (a,b)\in\eta\}= \{ ( (\lambda'_z)^{-1}(a),(\lambda'_z)^{-1}(b))\mid (a,b)\in\eta\}.$$
Let $(a,b)\in \eta_0$. Then there exists $c\in M$ such that
$c+a=c+b$. Hence
$c\circ(\lambda'_c)^{-1}(a)=c+a=c+b=c\circ(\lambda'_c)^{-1}(b)$.
Hence,
$$\lambda'_{c}\lambda'_{(\lambda'_c)^{-1}(a)}=\lambda'_{c\circ(\lambda'_c)^{-1}(a)}=\lambda'_{c\circ(\lambda'_c)^{-1}(b)}=\lambda'_{c}\lambda'_{(\lambda'_c)^{-1}(b)}$$
and thus
$$\lambda'_{(\lambda'_c)^{-1}(a)}=\lambda'_{(\lambda'_c)^{-1}(b)}.$$
Since $\eta_0=\{(\lambda'_c(a),\lambda'_c(b))\mid (a,b)\in
\eta_0\}$, we have that $\lambda'_a=\lambda'_b$, for all $(a,b)\in
\eta_0$. Because
$$\eta_n=\{(\lambda'_z(a),\lambda'_z(b))\mid (a,b)\in \eta_n\},$$ for
all non-negative integers $n$, it is easy to prove, by induction on
$n$, that $\lambda'_a=\lambda'_b$, for all $(a,b)\in \eta_n$. Hence
$\lambda'_a=\lambda'_b$, for all $(a,b)\in \eta$.

Let $(a,b)\in \eta$. Then $(\lambda'_c(a),\lambda'_c(b))\in \eta$.
Thus $(c\circ a,c\circ b)=(c+\lambda'_c(a),c+\lambda'_c(b))\in
\eta$. Since $\lambda'_a=\lambda'_b$, we have
$$(a\circ c,b\circ c)=(a+\lambda'_a(c),b+\lambda'_b(c))=(a+\lambda'_a(c),b+\lambda'_a(c))\in \eta.$$
Hence $\eta$ is a congruence on $(M,\circ)$, and the result follows.
\end{proof}

With the assumptions and notations as in  Example~\ref{ex1}, $\bar
M=M/\eta$. Let $ M\longrightarrow \bar M: a\mapsto \overline{a}$ be
the natural projection. Let $\bar\lambda\colon (\bar
M,\circ)\longrightarrow \Aut(\bar M,+)$ be the map defined by
$\bar\lambda(\bar a)=\bar\lambda_{\bar a}$ and $\bar\lambda_{\bar
a}(\bar b)=\overline{\lambda'_a(b)}$, for all $a,b\in M$.

Note that $\bar\lambda$ is well-defined, because if $\bar c=\bar a$
and $\bar d=\bar b$, then, by Lemma~\ref{congruence},
$\overline{\lambda'_a (b)} =\overline{\lambda'_a(d)}$ and
$\lambda'_a=\lambda'_c$, and
$$\overline{\lambda'_a(b)}=\overline{\lambda'_a(d)}=\overline{\lambda'_c(d)}.$$
Now it is easy to check that $\bar\lambda_{\bar a}\in \Aut(\bar
M,+)$ and that $\bar\lambda$ is a homomorphism such that $\bar
a\circ\bar b=\bar a+\bar\lambda_{\bar a}(\bar b)$, for all $a,b\in
M$.

\begin{remark} \label{remarkbijective}
{\rm If, furthermore,  the left non-degenerate set-theoretic solution $(X,r)$ is   finite and  bijective  then one can say more.
To do so, it is convenient to keep the notation $M=M(X,r)$ and $A=A(X,r)$. So $M\subseteq A\rtimes \im{\lambda'}$.
Jespers, Kubat and Van Antwerpen \cite[Proposition~2.9]{JKV} proved that there exists $t \geq 1$ and a central element $(z,1)\in M$, with  $z\in Z(A)$ and $g(z)=z$ for all $g\in \im (\lambda')$,
such that the least cancellative congruence on $(A,+)$ is
\begin{align*}
    \eta  &
    = \{ (a,b) \in A \times A \mid a + \underbrace{z + \cdots + z}_\text{i times} = b + \underbrace{z + \cdots + z}_\text{i times},\; \text{ for all }  i\geq t \}
    \\ &=
    \{ (a,b) \in A \times A \mid c+ a  = c+ b  \text{ for some } c \in A \}
    \\&=\eta_0
\end{align*}
Note that $(a,b) \in \eta$ implies that $\lambda'_a =\lambda'_b$. Hence, it follows
from Proposition~4.2 in \cite{JKV} that
the (least) cancellative congruence  on $(M,\circ)$ is
 $$\eta_M =\{ ((a,\lambda'_a),(b,\lambda'_b)) \mid (a,b)\in \eta \}.$$
It follows that the natural map
   $$M/\eta_M \longrightarrow (A/\eta ) \rtimes \im(\lambda')$$
i.e.
  $\overline{( a,\lambda'_a)} \mapsto (\bar a , \lambda'_a)$,
is an injective monoid homomorphism and $M/\eta_M$ is a regular
submonoid of $ (A/\eta ) \rtimes \im(\lambda')$. So  we obtain a
bijective 1-cocycle $(M/\eta_M , \circ) \longrightarrow (A/\eta
,+)$, with respect to $\overline{\lambda}$, that extends the mapping
$\overline{ (a,\lambda'_a)} \mapsto \overline{a}$. Because $r$ is
bijective we know (see explanation in Example~\ref{ex1}) that
$(A,+)$ consists of normal elements and thus $(A/\eta ,+)$ is a left
and right Ore monoid and also $(M/\eta_M,\circ)$ is a left and right
Ore monoid. Hence they both have a group of fractions, denoted
$\gr(A/\eta)$ and $\gr((M/\eta_M)$ respectively. It easily is
verified that $\gr (M/\eta_M)=G(X,r)$, the structure group of
$(X,r)$, $\gr(A/\eta)=G(X,r')$, the structure group of the derived
solution $(X,r')$, $\gr (M/\eta_M) \subseteq \gr(A/\eta) \rtimes \im
(\lambda')$, where by abuse of notation   $\lambda': \gr (M/\eta_M)
\rightarrow \Aut (A/\eta)$ is the natural extension of the mapping
$\bar \lambda$ and also $ \gr (M/\eta_M)$ is a regular subgroup of
$\gr(A/\eta) \rtimes \im (\lambda')$. The latter was proven by Lebed
and Vendramin in \cite[Theorem~3.4.]{LV} in case $(X,r)$ is
bijective, (left and right) non-degenerate and finite. }
\end{remark}

\begin{question}
If $(X,r)$ is a left non-degenerate solution of the YBE does there exist
a bijective 1-cocycle $(M/\eta_M , \circ) \longrightarrow (A/\eta,+)$,
with respect to $\overline{\lambda}$, that extends the  mapping $\overline {(a,\lambda'_a)} \mapsto \overline{a}$.
 In other words, can one avoid the bijective assumption in Remark~\ref{remarkbijective}?
\end{question}

Let $\bar \phi\colon \bar M\times\bar M\longrightarrow \bar M$ be
the map defined by $\bar\phi(\bar a,\bar b)=\bar\lambda_{\bar
a}(\bar b)$, for all $a,b\in M$. Then $(\bar M,+,\circ,\bar \phi)$
is a left semi-truss.

\begin{lemma}\label{normal}
Let $a,b\in M=M(X,r)$. Then there exists $c\in M$ such that $a+b=b+c$.
\end{lemma}

\begin{proof}
There exist non-negative integers $n,m$ and $x_1,\dots,
x_n,y_1,\dots,y_m\in X$ such that $a=x_1+\dots +x_n$ and
$b=y_1+\dots + y_m$. Clearly we may assume that $n,m$ are positive
integers. We shall prove the result by induction on $n+m$. If
$n=m=1$, then
$x_1+y_1=y_1+\sigma_{y_1}(\gamma_{\sigma^{-1}_{x_1}(y_1)}(x_1))$, by
the defining relations of $(M,+)$. Suppose that $m+n>2$, and that
the result is true for $m+n-1$. If $n>1$, then by the induction
hypothesis there exists $c'\in M$ such that $a+b=x_1+b+c'$, and by
the induction hypothesis again there exists $c''\in M$ such that
$x_1+b=b+c''$. Hence $a+b=b+c''+c'$, in this case. Suppose that
$n=1$. In this case $m>1$ and
$$a+b=x_1+b=y_1+\sigma_{y_1}(\gamma_{\sigma^{-1}_{x_1}(y_1)}(x_1))+y_2+\dots +y_m.$$
Hence, by the induction hypothesis, there exists $c\in M$ such that
$$\sigma_{y_1}(\gamma_{\sigma^{-1}_{x_1}(y_1)}(x_1))+y_2+\dots +y_m=y_2+\dots +y_m+c.$$
Thus $a+b=b+c$, in this case. Therefore the result follows by
induction.
\end{proof}

By Lemma~\ref{normal}, the left cancellative monoid $(\bar M,+)$ satisfies that
for all $\bar a,\bar b\in\bar M$,   there exists a unique $\bar c\in
\bar M$ such that $\bar a+\bar b=\bar b+\bar c$.  So,
the multiplicative monoid $(\bar M, \circ)$ is  left cancellative.
Hence, we have the following corollary.

\begin{corollary}  Let $(X,r)$ be a left non-degenerate set-theoretic solution of the YBE. Let $\eta$ be the left cancellative congruence on $(M(X,r'),+)$.
Then $(\bar M,+,\circ,\bar \phi)$ is a  left semi-truss   with $\bar M +\bar a \subseteq\bar a +\bar M$ for all $\bar a \in \bar M$ and it satisfies the conditions of Proposition~\ref{st2},
with $\bar\phi(\bar a,\bar b)=\bar\lambda_{\bar
a}(\bar b)$, for all $\bar a,\bar b\in \bar M$.
In particular,  $(\bar M,\bar r)$, where
$$\bar r(\bar a,\bar b)=(\bar \lambda_{\bar a}(\bar b),\bar \lambda^{-1}_{\bar \lambda_{\bar a}(\bar b)}(c(\bar a,\bar \lambda_{\bar a}(\bar b)))),$$
for all $\bar a,\bar b\in \bar M$, is a left non-degenerate
set-theoretic solution of the YBE.
In particular, $(\bar X, \bar r_{| \bar X})$ is a left non-degenerate solution on the image $\bar X$ of $X$ in $\bar M$.
\end{corollary}

We say that a left non-degenerate solution $(X,r)$ of the YBE is
injective if the natural map $X\rightarrow M/\eta$ is injective.
Obvious such examples are irretractable solutions, and in this case $r = \bar
r|_{{\bar X}^2}$. Note that if  $r$ also is bijective and
non-degenerate then  this notion corresponds with the one introduced
by Lebed and Vendramin in \cite{LV}. In \cite{LV} it also is shown
that, in this case, several properties of involutive solutions can
be generalized to injective ones.

\begin{corollary}
Any left non-degenerate injective set-theoretic solution $(X,r)$  of the YBE is the restriction of the induced
 left-non-degenerate solution of the YBE determined by a  left cancellative semi-truss $(M,+,\circ , \phi)$ with $M+a\subseteq a+M$ for all $a\in M$.
\end{corollary}

However,  note that  $(\bar M, \circ )$ is not necessarily the structure monoid of the solution
of $(\bar X , \bar r)$. Indeed, let $X=\Sym_3$ be the symmetric group of degree $3$. Let $(X,r)$ be the bijective non-degenerate solution defined by $r(a,b)=(aba^{-1},a)$ for all $a,b\in X$.
Note that the solution $(X,r)$ is non-involutive and  irretractable (because the center of $\Sym_3$ is trivial).
So, $X$ is naturally embedded in $(\bar M, \circ) =(M(X,r)/\eta ,\circ)$ and $\bar r_{|\bar X^{2}}=r$.
Let us denote the multiplication in the structure monoid  $M(X,r)$ by $\cdot$. In  $(M(X,r),\cdot )$
we  have
\begin{align*}
(1,2) \cdot (1,2,3) \cdot (1,2,3) \cdot (1,2,3) &= (1,3,2) \cdot (1,3,2) \cdot (1,3,2) \cdot (1,2)\\
&= (1,3,2) \cdot (1,3,2) \cdot (1,3) \cdot (1,3,2)\\
&= (1,3,2) \cdot (2,3) \cdot (1,3,2) \cdot (1,3,2) \\
&=(1,2) \cdot (1,3,2) \cdot (1,3,2) \cdot (1,3,2)
\end{align*}
while  $(1,2,3)\cdot (1,2,3) \cdot (1,2,3)\neq (1,3,2)\cdot (1,3,2) \cdot (1,3,2)$.
Hence, $M(X,r)$ is not left cancellative, while $\bar M$ is left cancellative.
Thus   $\bar M$ is not the structure monoid of $(X, r)$.

The following  problem remains a challenge.

\begin{question}
Determine when a left non-degenerate solution $(X,r)$ of the YBE is
cancellative injective. If $(X,r)$ is a left  non-degenerate solution
that is injective then does there exists a finite left cancellative
semi-truss in which $X$ can be embedded naturally? In case $r$ also
is finite, bijective and non-degenerate this has been proven by
Lebed and Vendramin in \cite{LV}.
\end{question}

\noindent {\bf Acknowledgement.} The authors would like to thank the referee
for providing  a sub- and superscript avoiding,  more elegant  and simple proof of Theorem~\ref{MainIrr} by  using the notation and a result of  Rump  proved in \cite{R}.

\vspace{30pt}
 \noindent \begin{tabular}{llllllll}
  Ferran Ced\'o && Eric Jespers\\
 Departament de Matem\`atiques &&  Department of Mathematics \\
 Universitat Aut\`onoma de Barcelona &&  Vrije Universiteit Brussel \\
08193 Bellaterra (Barcelona), Spain    && Pleinlaan 2, 1050 Brussel, Belgium \\
 cedo@mat.uab.cat && Eric.Jespers@vub.be \\ \\
 Charlotte Vermwimp && \\ Department of
Mathematics &&
\\  Vrije Universiteit Brussel &&\\
Pleinlaan 2, 1050 Brussel, Belgium  &&\\
Charlotte.Verwimp@vub.be &&
\end{tabular}

\end{document}